\documentclass[12pt]{amsart}

\usepackage{amssymb, verbatim}
\usepackage{enumerate}

\date{\today}

\newtheorem{teor}{Theorem}[section]
\newtheorem{propo}[teor]{Proposition}
\newtheorem{coro}[teor]{Corollary}

\theoremstyle{definition}
\newtheorem{defi}[teor]{Definition}
\theoremstyle{remark}
\newtheorem{rem}{Remark}

\newfont{\script}{eusm10 scaled\magstep1}

\usepackage{amssymb, verbatim}

\newcommand{\feriet}[8]{\,{}F^{#1}_{#2}\left(\!\!%
\begin{array}{c|c}
\begin{array}{cc}{\displaystyle{#3}:\displaystyle{#4}}\\[-0.1ex]
{\displaystyle{#5}:\displaystyle{#6}} \end{array} &
\,{\displaystyle{#7},\displaystyle{#8}}
\end{array} \right)}

\title[Bivariate partial differential equations and OP solutions] {Bivariate second--order
linear partial differential equations and orthogonal polynomial solutions}
\author[Area]{I. Area}
\author[Godoy]{E. Godoy}
\author[Ronveaux]{A. Ronveaux}
\author[Zarzo]{A. Zarzo}

\address[Area]{Departamento de Matem\'{a}tica Aplicada II, E.T.S.E. Telecomunicaci\'{o}n,
Universidade de Vigo, 36310--Vigo, Spain.}
\email[Area]{area@dma.uvigo.es}

\address[Godoy]{Departamento de Matem\'{a}tica Aplicada II, E.T.S. Ingenieros
Industriales, Universidade de Vigo, 36310--Vigo,
Spain.}\email[Godoy]{egodoy@dma.uvigo.es}

\address[Ronveaux]{Departement de Math\'{e}matique,
Universit\'{e} Catholique de Louvain, B\^atiment Marc de
Hemptinne, Chemin du Cyclotron 2, B--1348 Louvain--la--Neuve,
Belgium.} \email[Ronveaux]{andre.ronveaux@student.uclouvain.be}

\address[Zarzo]{Instituto Carlos I de F\'{\i}sica Te\'{o}rica y Computacional,
Facultad de Ciencias, Universidad de Granada, Spain and
Departamento de Matem\'{a}tica Aplicada, E.T.S. Ingenieros
Industriales, Universidad Polit\'{e}cnica de Madrid, Spain.}
\email[Zarzo]{azarzo@etsii.upm.es}



\begin{document}

\begin{abstract}
In this paper we construct the main algebraic and differential properties and the weight functions of orthogonal
polynomial solutions of bivariate second--order linear partial differential equations, which are admissible
potentially self--adjoint and of hypergeometric type. General formulae for all these properties are
obtained explicitly in terms of the polynomial coefficients of the partial differential equation, using
vector matrix notation. Moreover, Rodrigues representations for the polynomial eigensolutions
and for their partial derivatives of any order are given. Finally, as illustration, these results are
applied to specific Appell and Koornwinder orthogonal polynomials, solutions of the same partial differential equation.
\end{abstract}
\maketitle

\section{Introduction}

The theory of orthogonal polynomials in one variable is in permanent expansion due to its relationship with other areas of
mathematics and also with several applications in physics and engineering. They provide a natural way to solve many
types of important differential equations of mathematical physics, expanding solutions in appropriate Fourier series
of orthogonal polynomial basis. They play therefore an important role in the study of wave mechanics, heat conduction,
electromagnetic theory, quantum mechanics or mathematical statistics.

In this context, it is first relevant to study whether a (one variable) polynomial fa\-mi\-ly,
$\left\{p_n(x)\right\}_{n\in\mathbb{N}}$ ($x\in\mathbb{R}$), is orthogonal. This problem is solved in
different ways but the Favard's theorem \cite{SZE}, linking orthogonality and the fundamental three--term recurrence relation
\begin{equation}\label{ttrr1v}
x p_n(x) = \alpha_n p_{n+1}(x) + \beta_n p_{n}(x) + \gamma_n p_{n-1}(x)\,, \quad \gamma_{n} \neq 0,
\end{equation}
provides certainly the most powerful characterization \cite{CHI}. Here, if $p_n(x) = g_{n,n} x^n + g_{n,n-1} x^{n-1} + g_{n,n-2}x^{n-2} + \cdots$, then orthogonality of the $p_n$--family leads easily to the well known expressions (see e.g. \cite{CHI}, \cite{Koe}, \cite[equation (1.4.17), p. 14]{NSU} or \cite[p. 36]{NU}):
\begin{equation}\label{ttrrcof1v}
\alpha_n = \frac{g_{n,n}}{g_{n+1,n+1}}\,,\quad \beta_n = \frac{g_{n,n-1} - \alpha_n g_{n+1,n} }{g_{n,n}}\,,\quad \gamma_n = \frac{g_{n,n-2} - \beta_n g_{n,n-1} - \alpha_n g_{n+1,n-1}}{g_{n-1,n-1}}\,.
\end{equation}

It is also important to provide ways of constructing efficiently these polynomials. For, several approaches are at hand. Besides the use of the recurrence (\ref{ttrr1v}) itself, we can mention (in a non-exhaustive way) the generating function methods, those based on a Rodrigues formula  or the hypergeometric approach which gives nice and useful re\-pre\-sen\-ta\-tions of the polynomials in terms of hypergeometric series.

As it is very well known, among all the one variable orthogonal polynomials, the four classical continuous families of Jacobi, Laguerre, Hermite and Bessel, are those sharing the widest set of properties. Besides the three--term recurrence~(\ref{ttrr1v}) \cite{CHI,NU,SZE}, they can be characterized in a number of ways, e.g. they are orthogonal polynomial solutions of the  hypergeometric type differential equation \cite{BOCHNER,NU} and the $k$--th derivatives of each family are again orthogonal and belong to the same family \cite{AlSal90,NU}. Moreover, the orthogonality weight functions satisfy Pearson--type equations \cite{BOCHNER,NSU} giving rise to Rodrigues formulae \cite{AlSal90,NU} for the corresponding orthogonal polynomials and for their derivatives of any order. Also, the orthogonal polynomials satisfy a number of algebraic and differential properties such as derivative representations \cite{AlSal90,MBP} or also structure relations \cite{AlSal90,CHI,KOOR}, among other properties. The list of references in this paragraph is not exhaustive but only indicative of the kind of references that could be examined on this topic.

In these classical settings, it is remarkable that the coefficients appearing in all the aforementioned algebraic
and differential characterizations can be explicitly computed in terms of the polynomial
coefficients $\sigma(x)$ and $\tau(x)$ of the hy\-per\-geo\-me\-tric--type differential equation  \cite{BEZ1,GRZA,NU,BEZ2,YDZ94, YDN94, Z07}
\[
\sigma(x) y''(x) + \tau(x) y'(x) + \lambda_{n} y(x)=0\,, \qquad
\lambda_{n}=-n \tau'-\frac{1}{2}n(n-1)\sigma''.
\]

Our main contribution is to extend this remarkable property to the bivariate situation for polynomial solutions of admissible potentially
self--adjoint partial differential equations of hypergeometric--type \cite{KS, LYS1, LYS2}. These polynomials play an important role in many
applications, for instance in spectral/$hp$--finite element methods for solving partial differential equations \cite{DUBINER,KARNIADAKIS}.

One essential difference between polynomials in one variable and in several variables is the lack of an obvious basis in the latter \cite{DUXU}. One
possibility to avoid this problem is to consider graded lexicographical order and use the matrix vector representation, first
introduced by Kowalski \cite{KOW,KOW2} and afterwards considered by Xu \cite{XU93,XU94}. In fact, using this point of view, in \cite{AFPP}
the authors proved some structure and orthogonality relations for the successive partial derivatives of the vector orthogonal polynomials
associated with a quasi--definite moment functional which satisfies a Pearson--type partial differential equation.

In this paper we deal with bivariate polynomials written in vector representation (and graded lexicographical order) which
are solutions of admissible potentially self--adjoint linear second order partial differential equation of hypergeometric type. In this context, we
prove that (as it happens in the one variable hypergeometric type case) the coefficients characterizing the three--term recurrence relations, the
first structure relations and the derivative representations fulfilled by the vector polynomials can be written explicitly in terms of the
coefficients of the partial differential equation they satisfy. In the bivariate discrete case some results in this direction have been already given in \cite{RAG,RAG2,RAG3}.

The structure of the paper is as follows: In Section \ref{Sec:2},
after introducing basic definitions and notations, we present in
Proposition \ref{pearsonEDP} the general framework to be considered
through the paper, i.e. admissible potentially self--adjoint second
order partial differential equation of hypergeometric type. In
Section \ref{Sec:rodrigues} we give the partial differential
equations for the partial derivatives of the eigensolutions and we
construct the corresponding weight functions for the orthogonal
polynomials. Then, the relations linking these weight functions are
obtained and they allow us to deduce a Rodrigues formula for the
orthogonal polynomial solutions and for their partial derivatives of
any order. In Sections \ref{Sec:4} and \ref{Sec:5}, using vector
matrix notation \cite{DUXU, KOW, KOW2}, general formulae for the
main algebraic and differential properties (three--term recurrence
relations, structure relations and derivative representations) are
explicitly obtained in terms of the coefficients fully
characterizing the partial differential equation, both in monic and
non--monic cases. Finally, in Section \ref{Sec:6} our results are
applied to three different eigensolutions of the same partial
differential equation, which are orthogonal on the same domain with
respect to the same weight function $\varrho(x) = x^{\alpha-1}\,
y^{\beta-1}$ and so, they form biorthogonal families. First, the two
parameter monic Appell polynomials family \cite{ERDII} is analyzed
in detail. From that, for the two remaining cases (non--monic Appell
polynomials \cite{APP} and Koornwinder polynomials \cite{KOOR75})
the corresponding properties can be deduced in two ways: from the
results in Section \ref{Sec:4} or by means of their expansions in
terms of monic Appell polynomials which we establish here.

\section{Vector representation and admissible partial differential equations of hypergeometric type}\label{Sec:2}

Let $\textbf{x}=(x,y)\in\mathbb{R}^2$, and let $\textbf{x}^n$ ($n\in \mathbb{N}_0$) denote the column vector of the monomials $x^{n-k} y^{k}$, whose elements are arranged in graded lexicographical order (see \cite[p. 32]{DUXU}):
\begin{equation}\label{MONO}
\textbf{x}^n= (x^{n-k}y^{k})\,,\quad 0 \leq k \leq n, \quad  n\in \mathbb{N}_0\,.
\end{equation}
Let $\{P_{n-k,k}^n(x,y)\}$ be a sequence of polynomials in the space $\Pi_n^2$ of all polynomials of total degree at most $n$ in two variables, $\textbf{x}=(x,y)$, with real coefficients. Such polynomials are finite sums of terms of the form $ax^{n-k}y^{k}$, where $a \in \mathbb{R}$.

From now on the (column) vector representation \cite{KOW,KOW2} will be adopted, so that ${\mathbb{P}}_n$ will denote the (column) polynomial vector
\begin{equation}\label{defPn}
{\mathbb{P}}_n= (P_{n,0}^n(x,y), P_{n-1,1}^n(x,y), \dots,P_{1,n-1}^n(x,y), P_{0,n}^n(x,y))^\text{T}.
\end{equation}
Then, each polynomial vector ${\mathbb{P}}_n$ can be written in terms of the basis (\ref{MONO}) as:
\begin{equation}\label{EXPP}
{\mathbb{P}}_n= G_{n,n}\textbf{x}^n+ G_{n,n-1}\textbf{x}^{n-1}+ \dots + G_{n,0}\,\textbf{x}^0,
\end{equation}
where $G_{n,j}$ are matrices of size $(n+1)\times(j+1)$ and the leading matrix coefficient $G_{n,n}$ is a nonsingular square matrix of size $(n+1)\times(n+1)$.

\begin{defi}[Monic polynomial vector]\label{monic} A polynomial vector $\widehat{\mathbb{P}}_n$ is said to be monic if its leading matrix coefficient $\widehat{G}_{n,n}$ is the identity matrix (of size $(n+1)\times(n+1)$); i.e.:
\begin{equation}\label{EXPPMonico}
\widehat{\mathbb{P}}_n= \textbf{x}^n+ \widehat{G}_{n,n-1}\textbf{x}^{n-1}+ \dots + \widehat{G}_{n,0}\,\textbf{x}^0\,.
\end{equation}
Then, each of its polynomial entries $\widehat{P}_{n-k,k}^n(x,y)$ are of the form:
\begin{equation}\label{CompoMonico}
\widehat{P}_{n-k,k}^n(x,y) = x^{n-k} y^{k} + \text{terms of lower total degree}\,.
\end{equation}
In what follows the ``hat" notation $\widehat{\mathbb{P}}_n$ will represent monic polynomials.
\end{defi}

\begin{defi}[Orthogonality]\label{ortho}
Let $\mathcal{L}$ be a moment linear functional acting on $\Pi_n^2$. A sequence of polynomials $\{P_{n-k,k}^n(x,y)\} \subset \Pi_n^2 $ ($n \in \mathbb{N}_0$), is said to be orthogonal with respect to $\mathcal{L}$ or, equivalently, $\{{\mathbb{P}}_n\}_{n \geq 0}$ (as defined by Eqs.~(\ref{defPn})--(\ref{EXPP})) is a vector orthogonal polynomial family with respect to $\mathcal{L}$, if for each $n\in \mathbb{N}_0$ there exist an invertible matrix $H_n$ of size $n+1$ such that:
\begin{align}\label{ortog1}
&\mathcal{L} \left[ (\textbf{x}^m{\mathbb{P}}_n^T) \right]=0\in \mathcal{M}^{(m+1,n+1)}\,, \quad n >m,\\\label{ortog2}
&\mathcal{L} \left[ (\textbf{x}^n{\mathbb{P}}_n^T) \right]=H_n\in \mathcal{M}^{(n+1,n+1)}\,.
\end{align}
\end{defi}

An integral representation of this orthogonality functional
$\mathcal{L}$ can be constructed by means of a weight function
$\varrho := \varrho(x,y)$ defined in a certain domain $D\subset
\mathbb{R}^2$:
\begin{equation}\label{FUNCTIONAL0}
\mathcal{L}(P)=\iint_{D} P(x,y)\varrho(x,y)\,dx\,dy\,,\quad P \in \Pi_n^2\,,
\end{equation}
which is defined in the set $\Pi_n^2$ provided that all the above integrals exist. Then, the family $\{{\mathbb{P}}_n\}_{n \geq 0}$ is said to be orthogonal with respecto to $\varrho$ in the domain $D$.

In this multivariate context, Bochner \cite{BOCHNER} posed the problem of identifying those families of polynomials which are eigenfunctions of a second order linear partial differential operator. Krall and Sheffer \cite{KS} started to study eigenfunctions which are orthogonal over a domain giving conditions of admissibility and a first attempt of classifying admissible equations. Engelis \cite{Engelis} gave a detailed list of second order linear partial differential equations for which orthogonal polynomial in two variables are solutions. This question was afterwards studied and systematically described by Suetin \cite{SUE}. In this paper, we analyze admissible potentially self--adjoint partial differential equations of hypergeometric type.

In order to present this study, we consider a bivariate class of
linear partial differential equations, introduced as ``the basic
class" by Lyskova \cite{LYS2} for the multivariate case (see also
\cite{GRANADA2009}) and called here hypergeometric type equations:
\begin{multline}\label{original}
(a_1 x^2 + b_1x + c_1)\partial_{xx}u(x,y) + 2(a_3 xy + b_3x + c_3y + d_3)\partial_{xy}u(x,y)  \\
+ (a_2 y^2 + b_2y + c_2)\partial_{yy}u(x,y) + (e_1 x + f_1)\partial_{x}u(x,y) + (e_2 y + f_2)\partial_{y}u(x,y) + \lambda u(x,y) =0,
\end{multline}
where $a_j$, $b_j$, $c_j$, $d_j$ and $\lambda$ are real numbers. The solutions of this equation have the remarkable property that all the partial derivatives of any order of these solutions are also solutions of an equation of the same form.

Moreover, we shall also consider admissible partial differential equations.

\begin{defi}\label{admis}
A second order partial differential equation is admissible if and only if \cite{KS,SUE} for any non--negative integer $n$ there exists a number $\lambda_{n}$ such that equation (\ref{original}) with $\lambda := \lambda_n$ has $n+1$ linearly independent solutions which are polynomials of total degree $n$ and has no non--trivial solutions in the set of polynomials of total degree less than $n$.
\end{defi}

The following characterization has been proved in \cite{LYS2,SUE}.

\begin{propo}\label{admhyp}
Equation (\ref{original}) is an admissible second order partial differential equation of hypergeometric type if and only if it can be written in the form
\begin{align}\label{Lysk1}
&(ax^2 + b_1x + c_1)\partial_{xx}u(x,y) + 2(axy + b_3x + c_3y + d_3)\partial_{xy}u(x,y)  \\
&+ (ay^2 + b_2y + c_2)\partial_{yy}u(x,y) + (ex + f_1)\partial_{x}u(x,y) + (ey + f_2)\partial_{y}u(x,y) + \lambda_n u(x,y) =0,\nonumber
\end{align}
where $\lambda_n= -n((n-1)a+e)$ and the coefficients $a,b_j,c_j,d_3,e,f_j$ are arbitrary fixed real numbers, but the numbers $a$ and $e$ are such that the condition
\begin{equation}\label{NONU}
\varpi_{k} := ak+e \neq 0.
\end{equation}
holds true for any non--negative integer $k$.
\end{propo}

In these conditions an orthogonality weight function defined in certain domain of $\mathbb{R}^2$ which is related with the partial differential equation (\ref{Lysk1}) can be given for their orthogonal polynomial solutions as shown in \cite{SUE}.
\begin{propo}\label{pearsonEDP}
Let $\alpha(x,y)$ be the discriminant of Eq.~(\ref{Lysk1}), i.e.:
\begin{equation}\label{Eq:alpha}
\alpha(x,y)= \left( c_1 + x\,\left( b_1 + a\,x \right) \right) \,    \left( c_2 + y\,\left( b_2 + a\,y \right)\right) -\left( d_3 + b_3\,x + \left( c_3 + a\,x \right) \,y \right)^2,
\end{equation}
and $D\subset \mathbb{R}^2$ be the domain:
\begin{equation}\label{DETNN}
D = \left\{ (x,y) \in \mathbb{R}^2:\,\, \alpha(x,y) \neq 0\right\}\,.
\end{equation}
Defining the two functions \cite[eq. (15), p. 132]{SUE}
\begin{align}\label{FUNCWEI}
\beta(x,y)&= (-b_1 - c_3 + f_1 - 3 a x + e x)(ay^2 + b_2y + c_2)\\
& - (-b_2 - b_3 + f_2 - 3 a y +  e y)(axy + b_3x + c_3y + d_3)\,,\nonumber\\
\gamma(x,y)&= -(c_1 + x (b_1 + a x)) (b_2 + b_3 - f_2 + 3 a y - e y) \label{gggg} \\
&+ (b_1 + c_3 - f_1 + 3 a x - e x) (d_3 + b_3 x + (c_3 + a x) y)\,,\nonumber
\end{align}
assuming that in $D$ the following condition holds true:
\begin{equation}\label{potselfadjoint}
\frac{\partial}{\partial x} \left( \frac{\gamma(x,y)}{\alpha(x,y)} \right) = \frac{\partial}{\partial y} \left( \frac{\beta(x,y)}{\alpha(x,y)} \right)\,,
\end{equation}
considering the weight function or integrating factor of (\ref{Lysk1}) given by \cite[Equation (22), p. 134]{SUE}
\begin{equation}\label{WEIF}
\varrho(x,y)= \text{exp}
\left\{\int_{y_0}^{y}\frac{\gamma(x,y)}{\alpha(x,y)}dy +\int_{x_0}^{x}\left[\left(\frac{\beta(x,y)}{\alpha(x,y)}\right)_{y=y_0}\right]dx \right\},
\end{equation}
which determines (up to a multiplicative constant) the functional
\begin{equation}\label{FUNCTIONAL}
\mathcal{L}(P)=\iint_{D} P(x,y)\varrho(x,y)\,dx\,dy\,,\quad P \in \Pi_n^2\,,
\end{equation}
defined in the set $\Pi_n^2$ provided that all such integrals exist, then, there exists a unique monic vector polynomial family $\{\widehat{\mathbb{P}}_n\}_{n \geq 0}$ solution of~(\ref{Lysk1}) and orthogonal with respect to $\varrho$ in $D$, i.e. satisfying
\begin{equation}\label{ortorho}
\iint_{D} \mathbf{x}^m\,\, \widehat{\mathbb{P}}_n^{\,\,T} \varrho(x,y)\,dx\,dy\,=\,
\left\{\begin{array}{l}
0 \in \mathcal{M}^{(m+1,n+1)} \,,\,\, \mbox{if $n > m$,}\\
\phantom{blanco}\\
H_n \in \mathcal{M}^{(n+1,n+1)} \,,\,\, \mbox{if $m=n$,}
\end{array}\right.
\end{equation}
where $H_n$ $($of size $(n+1)\times (n+1)$$)$ is nonsingular.
\end{propo}

\begin{rem}
From the partial differential equation (\ref{Lysk1}) it is possible to introduce the linear operator ${\mathcal{D}}$ by
\begin{multline}\label{operatororiginal}
\mathcal{D}u=(a x^2 + b_1x + c_1)\partial_{xx}u + 2(a xy + b_3x + c_3y + d_3)\partial_{xy}u  \\
+ (a y^2 + b_2y + c_2)\partial_{yy}u + (e x + f_1)\partial_{x}u + (e y + f_2)\partial_{y}u,
\end{multline}
which is potentially self--adjoint if and only if (\ref{potselfadjoint}) holds true \cite[Theorem 1, p. 133]{SUE}.
\end{rem}

So, as it has been mentioned in the introduction, in this paper we shall be concerned with bivariate orthogonal polynomial families $\{{\mathbb{P}}_n\}_{n \geq 0}$ (in the sense of Definition~\ref{ortho}) written in vector representation (\ref{defPn})--(\ref{EXPP}), which are solutions of admisible potentially self--adjoint second order partial differential equations of hypergeometric type  characteri\-zed in the Proposition~\ref{admhyp}.

\section{Weight functions and Rodrigues formula for the polynomials and for their partial derivatives}\label{Sec:rodrigues}

By differentiating (\ref{Lysk1}) $(r+s)$ times, it turns out that
\[
z^{(r,s)}(x,y)=\frac{\partial^{r+s}u}{\partial x^{r}\partial y^{s}}(x,y), \quad r,s=0,1,2,\dots,
\]
satisfies an admissible second order partial differential equation of hypergeometric type
\begin{multline}\label{derLysk1}
(ax^2 + b_1x + c_1)\partial_{xx}z^{(r,s)}(x,y) + 2(axy + b_3x + c_3y + d_3)\partial_{xy}z^{(r,s)}(x,y)  \\
+ (ay^2 + b_2y + c_2)\partial_{yy}z^{(r,s)}(x,y) + \tau_x^{(r,s)}(x)\partial_{x}z^{(r,s)}(x,y) \\
+ \tau_y^{(r,s)}(y)\partial_{y}z^{(r,s)}(x,y) + \mu_{r+s} z^{(r,s)}(x,y) =0,
\end{multline}
where
\begin{align}
\tau_x^{(r,s)}(x)&= (e+2a(r+s))x + f_1+ r b_1+2 s c_3,\\
\tau_y^{(r,s)}(y)&= (e+2a(r+s))y + f_2+ 2rb_3 +sb_2, \\
\mu_{r+s}&= \lambda_n + (r+s) e + (r+s)(r+s-1)a,
\end{align}
and $a,b_j,c_j,d_3,e,f_j$ are arbitrary fixed real numbers
satisfying conditions (\ref{NONU}). Observe that equation
(\ref{derLysk1}) and the above relations have been obtained in the
multivariate case for hypergeometric type equations (not necessarily
admissible) in \cite[Theorem 1]{LYS2}.

Let us also introduce the linear operator
\begin{multline}\label{operdrs}
{\mathcal{D}}^{(r,s)} z=(ax^2 + b_1x + c_1)\partial_{xx}z+ 2(axy + b_3x + c_3y + d_3)\partial_{xy}z  \\
+ (ay^2 + b_2y + c_2)\partial_{yy}z + \tau_x^{(r,s)}(x)\partial_{x}z + \tau_y^{(r,s)}(y)\partial_{y}z.
\end{multline}

Applying condition (\ref{potselfadjoint}) to the operator
(\ref{operdrs}), we obtain that this operator is potentially
self--adjoint in a domain $D$, if and only if
\begin{equation}\label{PSADJ}
r \frac{\partial}{\partial x}\left(\frac{\omega(x,y)}{\alpha(x,y)}\right)+(s-r)
\frac{\partial}{\partial x}\left(\frac{1}{\alpha(x,y)}\frac{\partial \alpha(x,y)}{\partial y}\right)
-s\frac{\partial}{\partial y}\left(\frac{\theta(x,y)}{\alpha(x,y)}\right)=0,
\end{equation}
where
\begin{align}
\alpha(x,y)&= AC-B^2,\\
\omega(x,y)&=2A \frac{\partial B}{\partial x}-B\frac{\partial A}{\partial x}, \\
\theta(x,y)&= 2C \frac{\partial B}{\partial y}-B\frac{\partial
C}{\partial y},
\end{align}
and
\begin{align}
A=A(x,y)&= ax^2 + b_1x + c_1,\\
B=B(x,y)&=axy + b_3x + c_3y + d_3, \\
C=C(x,y)&= ay^2 + b_2y + c_2.
\end{align}

If the operator ${\mathcal{D}}^{(r,s)}$ defined in (\ref{operdrs}) is potentially self--adjoint in a domain $D$, there exists in this domain a positive and twice continuously differentiable function $\varrho^{(r,s)}(x,y)$ which is the solution of the system of differential equations (Pearson type equations) \cite[equations (7) and (8)]{LYS2} and \cite[p. 132]{SUE}
\begin{equation}\label{Peatypeq}
\left \{
\begin{matrix}
\displaystyle{\frac{1}{\varrho^{(r,s)}(x,y)}\frac{\partial \varrho^{(r,s)}(x,y)}{\partial x}= \frac{\beta^{(r,s)}(x,y)}{\alpha(x,y)}},\\
\displaystyle{\frac{1}{\varrho^{(r,s)}(x,y)}\frac{\partial \varrho^{(r,s)}(x,y)}{\partial y}=\frac{\gamma^{(r,s)}(x,y)}{\alpha(x,y)}},
\end{matrix}
\right.
\end{equation}
where
\begin{align}
\beta^{(r,s)}(x,y)&=\beta(x,y)+r \frac{\partial \alpha}{\partial x}(x,y)+s \theta(x,y),\\
\gamma^{(r,s)}(x,y)&=\gamma(x,y)+r \omega(x,y)+s \frac{\partial \alpha}{\partial y}(x,y),
\end{align}
and the polynomials $\beta(x,y)$ and $\gamma(x,y)$ have been defined in (\ref{FUNCWEI}) and (\ref{gggg}) respectively.

From the Pearson type equations (\ref{Peatypeq}) we obtain
\begin{equation}\label{pesoderiv}
\varrho^{(r,s)}(x,y)=exp \left\{\int_{y_0}^{y}\frac{\gamma^{(r,s)}(x,y)}{\alpha(x,y)} dy +\int_{x_0}^{x}\left[\left(\frac{\beta^{(r,s)}(x,y)}{\alpha(x,y)}\right)_{y=y_0}\right]dx
\right\},
\end{equation}
up to a multiplicative constant.

Now, we are in a position to establish the relation linking $\varrho^{(r,s)}(x,y)$ and $\varrho(x,y)$.

\subsection{Relation between weight functions}\label{Sec:3}

Now, we can establish the connection between $\varrho^{(r,s)}(x,y)$ and $\varrho^{(0,0)}(x,y) \equiv \varrho(x,y)$, given in (\ref{pesoderiv}) and (\ref{WEIF}) respectively. From equation (\ref{pesoderiv}), after straightforward computations, we obtain that
\begin{equation}\label{phirs}
\varrho^{(r,s)}(x,y)=\phi^{(r,s)}(x,y) \varrho(x,y), \qquad r,s=0,1,2,\dots,
\end{equation}
up to a multiplicative constant, where $\phi^{(r,s)}(x,y)$ is a polynomial which explicit expression depends on the coefficients of the partial differential equation (\ref{Lysk1}). After solving the non linear system of equations (\ref{PSADJ}) for any $r$ and $s$, we can reduce the solutions of the system to the following ten cases:
\begin{enumerate}[(i)]

\item If $b_1=2 c_3$ and $b_2=2 b_3$, we have
\[
\phi^{(r,s)}(x,y)=[\alpha(x,y)]^{r+s},
\]
where
\[
\alpha(x,y)=-(d_3+b_3 x + (c_3+a x) y)^{2} + (c_1 + x (2 c_3 + a x))(c_2 + y(2 b_3 + a y)).
\]

\item If $c_3 \neq 0$, $d_3 \neq 0$, $b_3 \neq 0$, $a=\frac{b_3 c_3}{d_3}$, $c_1=\frac{(b_1-c_3)d_3}{b_3}$, and $c_2=\frac{(b_2-b_3)d_3}{c_3}$, we have
\[
\phi^{(r,s)}(x,y)=\frac{[\alpha(x,y)]^{r+s}}{(d_3+c_3 y)^{r}(d_3 + b_3 x)^{s}},
\]
where
\begin{multline*}
\alpha(x,y)=-\frac{1}{b_3 c_3 d_3} \\
\times ((d_3 + b_3 x)(d_3 + c_3 y)(-b_1 (b_2 d_3 - b_3 d_3 + b_3 c_3 y) + c_3 (b_2 (d_3 - b_3 x) + 2 b_3 (b_3 x + c_3 y)))).
\end{multline*}

\item If $a=b_1=c_1=c_3=0$, we obtain
\[
\phi^{(r,s)}(x,y)=[\alpha(x,y)]^{r}, \qquad \alpha(x,y)=(d_3+b_3 x)^{2}.
\]

\item If $a=b_2=b_3=c_2=0$, we have
\[
\phi^{(r,s)}(x,y)=[\alpha(x,y)]^{s}, \qquad \alpha(x,y)=(d_3+c_3 y)^{2}.
\]

\item If $a=b_3=c_3=d_3=0$, we have
\[
\phi^{(r,s)}(x,y)=\frac{[\alpha(x,y)]^{r+s}}{(c_1+b_1 x)^{s} (c_2+b_2 y)^{r}} , \qquad \alpha(x,y)=(c_1+b_1 x)(c_2+b_2 y).
\]

\item If $a \neq 0$, $b_3=c_2=d_3=0$, and $c_1=\frac{(b_1-c_3)c_3}{a}$, we obtain
\[
\phi^{(r,s)}(x,y)=\frac{[\alpha(x,y)]^{r+s}}{y^{r} (c_3+ax)^{s}},
\]
where
\[
\alpha(x,y)=\frac{(c_3+ax) y (b_2 (b_1-c_3+ax)+a(b_1-2 c_3)y)}{a}.
\]

\item  If $c_3 \neq 0$, $a=b_3=0$, $b_1=c_3$, and $c_2=\frac{b_2 d_3}{c_3}$, we obtain
\[
\phi^{(r,s)}(x,y)=\frac{[\alpha(x,y)]^{r+s}}{(d_3+c_3y)^{r}},
\]
where
\[
\alpha(x,y)=\frac{(d_3+c_3 y)(b_2(c_1+c_3 x)-c_3 (d_3+c_3 y))}{c_3}.
\]

\item  If $b_3 \neq 0$, $a=c_3=0$, $b_2=b_3$, and $c_1=b_1 d_3/b_3$, we obtain
\[
\phi^{(r,s)}(x,y)=\frac{[\alpha(x,y)]^{r+s}}{(d_3+b_3 x)^{s}},
\]
where
\[
\alpha(x,y)=\frac{(d_3+b_3 x)(-b_3(d_3+b_3)x+b_1(c_2+b_3 y))}{b_3}.
\]

\item  If $a \neq 0$, $c_1=c_3=d_3=0$, and $c_2=\frac{(b_2-b_3)b_3}{a}$, we obtain
\[
\phi^{(r,s)}(x,y)=\frac{[\alpha(x,y)]^{r+s}}{x^{s} (b_3+a y)^{r}},
\]
where
\[
\alpha(x,y)=\frac{x(b_3+ay)(a(b_2-2 b_3)x+b_1(b_2-b_3+a y))}{a}.
\]

\item  If $c_1=c_2=d_3=b_3=c_3=0$, we obtain
\[
\phi^{(r,s)}(x,y)=\frac{[\alpha(x,y)]^{r+s}}{x^{s} y^{r}},
\]
where
\[
\alpha(x,y)=x y (a b_2 x+b_1(b_2+ay)).
\]

\end{enumerate}

\begin{rem}
Observe that in all the above cases the polynomial $\phi^{(r,s)}(x,y)$ defined in (\ref{phirs}) can be factorized as
\begin{equation}\label{Eq:phijjj}
\phi^{(r,s)}(x,y) = [\phi^{(1,0)}(x,y)]^{r} [\phi^{(0,1)}(x,y)]^{s},
\end{equation}
for any $r$ and $s$.
\end{rem}

\begin{rem}
It is important to notice here that different cases could give rise to the same $\phi^{(r,s)}(x,y)$ associated with the same partial differential equation. This situation appears in the example studied in detail in Section \ref{Sec:6} where the coefficients of the partial differential equation (\ref{equationappell}) fulfill the conditions of cases (vi), (ix) and (x).
\end{rem}

\begin{rem}
We must now mention that in \cite[Theorem 3]{LYS2}, Lyskova presented the equation (\ref{phirs}) in a non--explicit form.
\end{rem}

\subsection{Rodrigues formula} One of the main problems in the theory of orthogonal polynomials in several variables is to
obtain explicit expressions for the orthogonal polynomial solutions of the partial differential equation. In this direction
we could mention the works of Engelis \cite{Engelis}, who derived the Rodrigues formula for some classes of orthogonal
polynomials in two variables, and Suetin \cite[Theorem 3, p. 151]{SUE}, who showed that this Rodrigues
representation is one of the ways of constructing explicitly orthogonal polynomial families in the the potentially self--adjoint case.

As a consequence, in the context considered here (admissible
potentially self--adjoint partial differential equations of
hypergeometric type), we have the following explicit Rodrigues
formula for the polynomial solutions of (\ref{Lysk1}) of total
degree $n+m$
\begin{equation}\label{Eq:Rodrigues}
P_{n,m}(x,y)=\frac{\aleph_{n,m}}{\varrho(x,y)} \frac{\partial^{n+m}}{\partial x^{n} \partial y^{m}} \left[ \varrho(x,y) \left[ \phi^{(1,0)}(x,y) \right]^{n} \left[ \phi^{(0,1)}(x,y) \right]^{m} \right] \,,
\end{equation}
where $\aleph_{n,m}$ is a normalizing constant and the polynomials $\phi^{(r,s)}(x,y)$ have been introduced in (\ref{phirs}) and (\ref{Eq:phijjj}).

Having in mind that the partial differential equation for the partial derivatives  (\ref{derLysk1}) is also of hypergeometric type with the same discriminant $\alpha(x,y)$ defined in (\ref{Eq:alpha}), we also have a Rodrigues representation for the partial derivatives of any order, polynomial solutions of (\ref{derLysk1}), given by
\begin{equation}\label{Eq:Rodrigues-derivatives}
P_{n,m}^{(r,s)}(x,y)=\frac{\aleph_{n,m,r,s}}{\varrho^{(r,s)}(x,y)} \frac{\partial^{n+m-r-s}}{\partial x^{n-r} \partial y^{m-s}} \left[ \varrho^{(r,s)}(x,y) \left[ \phi^{(1,0)}(x,y) \right]^{n-r} \left[ \phi^{(0,1)}(x,y) \right]^{m-s} \right] \,,
\end{equation}
where $\varrho^{(r,s)}(x,y)$ is given in (\ref{pesoderiv})  and
$\aleph_{n,m,r,s}$ is a normalizing constant. In this way we have
obtained a natural extension to the bivariate case of the Rodrigues
representation for classical orthogonal polynomials in one variable
\cite{NSU}.

\section{Explicit expressions for algebraic properties}\label{Sec:4}

Let us consider a vector polynomial family $\{\mathbb{P}_n\}_{n\in\mathbb{N}_0}$ solution of~(\ref{Lysk1}) orthogonal in the sense of Proposition~\ref{pearsonEDP}, i.e. it is orthogonal with respect to a weight~(\ref{WEIF}) and satisfies~(\ref{ortorho}) in an appropriate domain
$D\in\mathbb{R}^2$. In such a conditions it can be proved that the family satisfy a number of algebraic and differential properties. Here we focus our attention in three of the most relevant: the three--term recurrence relations, the structure relations and the derivative representations. As already mentioned, our aim in this section is to give all of these relations in terms of the matrix coefficients in the expansions~(\ref{EXPP}) of the $\{\mathbb{P}_n\}$--family elements.

For, let us first introduce the matrices $L_{n,j}$ of size $(n+1)
\times (n+2)$ defined by
\begin{equation}\label{DEFLM}
L_{n,1}\textbf{x}^{n+1}=x\, \textbf{x}^n, \quad L_{n,2}\textbf{x}^{n+1}=y\, \textbf{x}^n,
\end{equation}
so that,
\begin{equation}\label{defLL}
\begin{array}{ll} L_{n,1}=\begin{pmatrix}
1 & & \text{\circle{10}}&0
\\&\ddots &&\vdots  \\&\text{\circle{10}}&1&0 \end{pmatrix}  &
\text{and} \quad
L_{n,2}=\begin{pmatrix} 0 & 1&
&\text{\circle{10}}
\\\vdots &  &\ddots &  \\0&\text{\circle{10}}&&1 \end{pmatrix}
\end{array}.
\end{equation}

Let us observe that
\begin{gather}
x^2\, \textbf{x}^n =L_{n,1}L_{n+1,1}\textbf{x}^{n+2}, \quad y^2\, \textbf{x}^n =L_{n,2}L_{n+1,2}\textbf{x}^{n+2}\,, \label{LLLL} \\
L_{n,2} L_{n+1,1} = L_{n,1}L_{n+1,2}, \nonumber
\end{gather}
and for $j=1,2$,
\begin{equation}
L_{n,j}\,L_{n,j}^{\text{T}} = I_{n+1},
\end{equation}
where $I_{n+1}$ denotes the identity matrix of size $n+1$.

Moreover, for $n \geq 1$,
\begin{equation}\label{DERMON}
\left \{
\begin{matrix}
\displaystyle{\partial_{x}\textbf{x}^n=\mathbb{E}_{n,1}\,\textbf{x}^{n-1}},\\
\displaystyle{\partial_{y}\textbf{x}^n=\mathbb{E}_{n,2}\,\textbf{x}^{n-1}},
\end{matrix}
\right.
\end{equation}
where the matrices $\mathbb{E}_{n,j}$ of size $(n+1) \times n$ are given by
\begin{equation}\label{DIAGDIF}
\begin{array}{rr} \mathbb{E}_{n,1}=\begin{pmatrix}
0 & n & & &\text{\circle{10}} \\
0 &  &n-1& &  \\
\vdots & &  &\ddots &  \\
0 & &\text{\circle{10}}& & 1 \\
0 & 0 &\dots&0&0\end{pmatrix}  & \text{and} \quad
\mathbb{E}_{n,2}=\begin{pmatrix} 0&\dots& &0 \\1 & &
&\text{\circle{10}}
\\ & 2 & &  \\& &\ddots &  \\&\text{\circle{10}}& &n
\end{pmatrix}\,.
\end{array}
\end{equation}

\subsection{The three--term recurrence relations}
The existence of a recurrence relation for a vector orthogonal polynomial family can be established in more general settings than those considered here. More precisely, the following existence theorem is proved in  \cite{DUXU}.

\begin{teor}
Let $\mathcal{L}$ be the positive definite moment functional as defined in (\ref{FUNCTIONAL0}) and $\{\mathbb{P}_{n}\}_{n \geq 0}$ be an orthogonal family with respect to $\mathcal{L}$. Then, for $n \geq 0$, there exist unique matrices $A_{n,j}$ of size $(n+1) \times (n+2)$, $B_{n,j}$ of size $(n+1) \times (n+1)$,and $C_{n,j}$ of size $(n+1) \times n$, such that
\begin{equation}\label{RRTT}
x_j\mathbb{P}_n=A_{n,j}\mathbb{P}_{n+1} + B_{n,j}\mathbb{P}_{n} + C_{n,j}\mathbb{P}_{n-1}, \quad j =1, 2
\end{equation}
with the initial conditions $\mathbb{P}_{-1}=0$ and $\mathbb{P}_{0}=1$. Here the notation $x_{1}=x$, $x_{2}=y$ is used.
\end{teor}
Now it is possible to generalize the well--known expressions (\ref{ttrrcof1v}) for the one variable case to the bivariate case. This is done
in the following proposition which is proved with the help of the auxiliary matrices $L_{n,j}$ defined in~(\ref{DEFLM})--(\ref{defLL}).

\begin{teor}\label{TTRRenG}
The explicit expressions of the matrices $A_{n,j}$, $B_{n,j}$ and $C_{n,j}$ ($j =1, 2$) appearing in (\ref{RRTT}) in terms of the values of the leading coefficients $G_{n,n}$, $G_{n,n-1}$ and $G_{n,n-2}$ in the expansions~(\ref{EXPP}) are given by
\begin{equation}\label{COEFTTR}
\left \{
\begin{matrix}
A_{n,j}=G_{n,n}L_{n,j}G_{n+1,n+1}^{-1}, \quad n \geq 0,\\
B_{0,j}=-A_{0,j}G_{1,0},  \\
B_{n,j}=(G_{n,n-1}L_{n-1,j}-A_{n,j}G_{n+1,n})G_{n,n}^{-1}, \quad n \geq 1,  \\
C_{1,j}=-(A_{1,j}G_{2,0}+B_{1,j}G_{1,0}),  \\
C_{n,j}=(G_{n,n-2}L_{n-2,j}-A_{n,j}G_{n+1,n-1}-B_{n,j}G_{n,n-1})G_{n-1,n-1}^{-1}, \quad n \geq 2\,.
\end{matrix}
\right.
\end{equation}
\end{teor}
\begin{proof} In equation (\ref{RRTT}), substitute $\mathbb{P}_n$ as given in (\ref{EXPP}), equate the coefficients of $\mathbf{x}^k$ for $k=n$, $n-1$, $n-2$ and solve the corresponding linear system.
\end{proof}

The above result is valid for any orthogonal polynomial sequence
(\ref{ortorho}). From now on we shall consider algebraic properties
involving the partial derivatives of the orthogonal polynomials. To
set up sufficient conditions for the existence of this type of
relations is not enough with the orthogonality of the vector
polynomial family. For these relations to exist, the orthogonal
polynomial family has to be solution of an admissible potentially
self--adjoint second order partial differential equation of
hypergeometric type (\ref{Lysk1}). With these assumptions, next we
present explicit expressions for the three--term recurrence relation
of the first partial derivatives, structure relations and derivative
representations.

\subsection{The three--term recurrence relations for the first partial derivatives}
Let us define, for $j=1,2$
\begin{equation}\label{notderxsin1}
{\mathbb{Q}}_n^{(j)}=L_{n,j}\frac{\partial}{\partial x_j}{\mathbb{P}}_{n+1}
= (\frac{\partial}{\partial x_j}P_{n+1,0}^{n+1}(x,y),
\frac{\partial}{\partial x_j}P_{n,1}^{n+1}(x,y), \dots,
\frac{\partial}{\partial x_j}P_{1,n}^{n+1}(x,y))^\text{T}\,,
\end{equation}
where ${\mathbb{P}}_n$ is defined in (\ref{defPn})--(\ref{EXPP}) and $L_{n,j}$ are defined in (\ref{defLL}). Here the notation $x_{1}=x$, $x_{2}=y$ is used. Thus, we have
\begin{equation}\label{EXPDER}
\frac{\partial}{\partial x_j}{\mathbb{P}}_{n+1}= Gd_{n,n}^{(j)}\textbf{x}^n+ Gd_{n,n-1}^{(j)}\textbf{x}^{n-1}+ \dots + Gd_{n,0}^{(j)}\,\textbf{x}^0,
\end{equation}
where
\begin{equation}\label{EXPDERCOEFS}
Gd_{n,k}^{(j)}= G_{n+1,k+1}\mathbb{E}_{k+1,j}\,,\quad 0 \leq k \leq n, \quad j=1,2,
\end{equation}
are matrices of size $(n+2)\times (k+1)$.

Let $\mathcal{L}$ be the positive definite moment functional as defined in (\ref{FUNCTIONAL}) and $\{\mathbb{P}_{n}\}_{n \geq 0}$ be an orthogonal family with respect to $\mathcal{L}$, so that it satisfies the admissible partial differential equation of hypergeometric type (\ref{Lysk1}).

Then, for $n \geq 0$ and $j=1,2$, there exist matrices $A_{n,j}^{(j)}$ of size $(n+2) \times (n+3)$, $B_{n,j}^{(j)}$ of size $(n+2) \times (n+2)$, and $C_{n,j}^{(j)}$ of size $(n+2) \times (n+1)$, such that
\begin{equation}\label{RRTTDER2}
x_j\frac{\partial}{\partial
x_j}{\mathbb{P}}_{n+1}=A_{n,j}^{(j)}\frac{\partial}{\partial x_j}{\mathbb{P}}_{n+2}
+ B_{n,j}^{(j)}\frac{\partial}{\partial x_j}{\mathbb{P}}_{n+1}
+ C_{n,j}^{(j)}\frac{\partial}{\partial x_j}{\mathbb{P}}_{n}, \quad j =1, 2,
\end{equation}
with the initial conditions $\frac{\partial}{\partial x_j}{\mathbb{P}}_{0}=0$ and $\frac{\partial}{\partial x_j}{\mathbb{P}}_{1} =G_{1,1} \mathbb{E}_{1,j}$.

The explicit expressions of the matrices $A_{n,j}^{(j)}$,
$B_{n,j}^{(j)}$ and $C_{n,j}^{(j)}$ can be obtained in terms of the
values of the leading coefficients $Gd_{n,n}^{(j)}$,
$Gd_{n,n-1}^{(j)}$ and $Gd_{n,n-2}^{(j)}$, defined in
(\ref{EXPDERCOEFS}), using the Proposition \ref{pearsonEDP}, but
they are not unique.

Since
\begin{equation}\label{relppqq}
\frac{\partial}{\partial x_j}{\mathbb{P}}_{n+1}=L^\text{T}_{n,j}{\mathbb{Q}}_n^{(j)}, \qquad n \geq 0,
\end{equation}
for $j=1,2$ we can write a three--term recurrence relation for the
${\mathbb{Q}}_n^{(j)}$ polynomials, which is now unique after
multiplying by the $L^\text{T}_{n,j}$ matrices. In this way, the
following theorem can be proved.

\begin{teor}
Let $\mathcal{L}$ be the positive definite moment functional as defined in (\ref{FUNCTIONAL}) and $\{\mathbb{P}_{n}\}_{n \geq 0}$ be
an orthogonal family with respect to $\mathcal{L}$. Then, for $n \geq 0$ and $j=1,2$, there exist unique matrices $\tilde{A}_{n,j}^{(j)}$ of
size $(n+1) \times (n+2)$, $\tilde{B}_{n,j}^{(j)}$ of size $(n+1) \times (n+1)$,and $\tilde{C}_{n,j}^{(j)}$ of size $(n+1) \times n$, such that
\begin{equation}\label{RRTTQQDER2}
x_j{\mathbb{Q}}_n^{(j)}=\tilde{A}_{n,j}^{(j)}{\mathbb{Q}}_{n+1}^{(j)} +
\tilde{B}_{n,j}^{(j)}{\mathbb{Q}}_n^{(j)} +
\tilde{C}_{n,j}^{(j)}{\mathbb{Q}}_{n-1}^{(j)}, \quad j =1, 2,
\end{equation}
where
\begin{align}\label{EXPRRTTDER}
\tilde{A}_{n,j}^{(j)}&= L_{n,j}A_{n,j}^{(j)}L^\text{T}_{n+1,j},\\
\tilde{B}_{n,j}^{(j)}&=L_{n,j}B_{n,j}^{(j)}L^\text{T}_{n,j}, \\
\tilde{C}_{n,j}^{(j)}&=L_{n,j}B_{n,j}^{(j)}L^\text{T}_{n-1,j},
\end{align}
with the initial conditions ${\mathbb{Q}}_{-1}^{(j)}=(0)$ and ${\mathbb{Q}}_0^{(j)}=L_{0,j}Gd_{1,1}^{(j)}$. Here the notation $x_{1}=x$, $x_{2}=y$ is used.
\end{teor}

\subsection{First structure relations} In the one variable case, the so--called first structure relation plays an important role since, e.g. it gives rise to lowering and rising operators. Next we present the extension to the bivariate case.
\begin{teor}\label{TEOSRTRR}
Let $\{\mathbb{P}_n\}_{n\in \mathbb{N}_0}$ be a vector orthogonal polynomial family satisfying Proposition~\ref{pearsonEDP}. Then, for $n \geq 1$, there exist unique matrices $W_{n,j}$ of size $(n+1) \times (n+2)$, $S_{n,j}$ of size $(n+1) \times (n+1)$,and $T_{n,j}$ of size $(n+1) \times
n$, such that
\begin{equation}\label{SRTRR}
\phi_j(x,y)\frac{\partial}{\partial x_j}\mathbb{P}_n=W_{n,j}\mathbb{P}_{n+1}+S_{n,j}\mathbb{P}_{n}+T_{n,j}\mathbb{P}_{n-1}, \quad j =1, 2,
\end{equation}
where
\begin{equation}\label{NOT:1}
\phi_{1}(x,y)=\phi^{(1,0)}(x,y)\,,\qquad \phi_{2}(x,y)=\phi^{(0,1)}(x,y)\,,
\end{equation}
and the polynomials $\phi^{(r,s)}(x,y)$ have been introduced in (\ref{phirs}) and (\ref{Eq:phijjj}). With the notation
\begin{equation}\label{Eq:notationphi}
\phi_{j}(x,y)=\alpha_j x^2 + \beta_{j} x y + \gamma_{j} y^{2} + \delta_{j} x + \varepsilon_{j} y + \omega_{j}, \quad j=1,2,
\end{equation}
the coefficients of the structure relations are explicitly given by
\begin{align*}
W_{n,j} G_{n+1,n+1}&= G_{n,n} \mathbb{E}_{n,j} \left( \alpha_j L_{n-1,1}L_{n,1} + \beta_{j} L_{n-1,1}L_{n,2} + \gamma_{j} L_{n-1,2} L_{n,2} \right), \\
S_{n,j} G_{n,n} &= G_{n,n} \mathbb{E}_{n,j} \left( \delta_{j} L_{n-1,1} + \varepsilon_{j} L_{n-1,2} \right) - W_{n,j} G_{n+1,n} \\
& + G_{n,n-1} \mathbb{E}_{n-1,j} \left( \alpha_{j} L_{n-2,1}L_{n-1,1} + \beta_{j} L_{n-2,2} L_{n-1,1} + \gamma_{j} L_{n-2,2} L_{n-1,2} \right), \\
T_{n,j} G_{n-1,n-1}&= \omega_{j} G_{n,n}  \mathbb{E}_{n,j} - W_{n,j} G_{n+1,n-1} - S_{n,j} G_{n,n-1} \\
&+ G_{n,n-1}  \mathbb{E}_{n-1,j} \left( \delta_{j} L_{n-2,1} + \varepsilon_{j} L_{n-2,2} \right) \\
&+ G_{n,n-2}  \mathbb{E}_{n-2,j} \left( \alpha_{j} L_{n-3,1} L_{n-2,1} + \beta_{j} L_{n-3,2} L_{n-2,1} + \gamma_{j} L_{n-3,2} L_{n-2,2} \right)  .
\end{align*}
\end{teor}
\begin{proof}
Since
\[
\phi_j(x,y)\frac{\partial}{\partial x_j}\mathbb{P}_n
\]
are polynomials of total degree $n+1$ we can write for $j=1,2$,
\[
\phi_j(x,y)\frac{\partial}{\partial x_j}\mathbb{P}_n = \sum_{\ell=0}^{n+1} \Lambda_{n,\ell} \mathbb{P}_\ell, \quad n \geq 1.
\]
As a consequence of Section \ref{Sec:2}, the polynomials $\frac{\partial}{\partial x}\mathbb{P}_n$ are orthogonal with respect to $\phi_{1}(x,y) \varrho(x,y)$, and $\frac{\partial}{\partial y}\mathbb{P}_n$ are orthogonal with respect to $\phi_{2}(x,y) \varrho(x,y)$. If we multiply from the right the latter equation by $\mathbb{P}_n^{\text{T}}$ and apply the corresponding orthogonality we get $\Lambda_{n,\ell}=0$ for $\ell < n-1$.

The explicit expressions for the matrix coefficients are obtained from equation (\ref{SRTRR}), by substituting $\mathbb{P}_n$ as given in (\ref{EXPP}), equating the coefficients of $\mathbf{x}^k$ for $k=n$, $n-1$, $n-2$ and solving the corresponding linear system.
\end{proof}

Observe that if we know the explicit expression of the polynomials $\phi_{j}(x,y)$, $j=1,2$, by using the results given in Section \ref{Sec:3}, it is possible to obtain explicitly the matrices $W_{n,j}$, $S_{n,j}$ and $T_{n,j}$ in (\ref{SRTRR}) in the following way: substitute $\mathbb{P}_n$ as given in (\ref{EXPP}) equate the coefficients of $\mathbf{x}^k$ for $k=n$, $n-1$, $n-2$ and solve the corresponding linear system, by using (\ref{DEFLM}), (\ref{LLLL}) and (\ref{DERMON}).

\subsection{Derivative representations (or second structure relations)}
Next we present finite--type relations between the orthogonal polynomial sequence $\{\mathbb{P}_n\}$ and the sequence of the partial derivatives $\{ \frac{\partial}{\partial x_j}\mathbb{P}_n \}$.
\begin{teor}
Let $\{\mathbb{P}_n\}_{n\in \mathbb{N}_0}$ be a vector orthogonal polynomial family satisfying Proposition~\ref{pearsonEDP}. Then, for $n \geq 2$ we have
\begin{equation}\label{POLDER1}
\mathbb{P}_n=V_{n,j}\frac{\partial}{\partial x_j}\mathbb{P}_{n+1} +Y_{n,j}\frac{\partial}{\partial x_j}\mathbb{P}_n
+Z_{n,j}\frac{\partial}{\partial x_j}\mathbb{P}_{n-1},
\end{equation}
where the matrices $V_{n,j}$ of size $(n+1) \times (n+2)$, $Y_{n,j}$ of size $(n+1) \times (n+1)$, and $Z_{n,j}$ of size $(n+1) \times n$ are given by
\begin{align*}
V_{n,j}&=A_{n,j}-A_{n-1,j}^{(j)} \,, \\
Y_{n,j}&=B_{n,j}-B_{n-1,j}^{(j)} \,, \\
Z_{n,j}&=C_{n,j}-C_{n-1,j}^{(j)} \,,
\end{align*}
and the matrices $A_{n,j}$, $B_{n,j}$ and $C_{n,j}$ are given in (\ref{COEFTTR}) and the matrices $A_{n-1,j}^{(j)}$, $B_{n-1,j}^{(j)}$ and $C_{n-1,j}^{(j)}$ are introduced in (\ref{RRTTDER2}).
\end{teor}
\begin{proof}
The above result is a consequence of (\ref{RRTT}) and (\ref{RRTTDER2}).
\end{proof}
Moreover, from (\ref{relppqq}) we can deduce the following relation:
\begin{equation}
\mathbb{P}_n=V_{n,j} L^\text{T}_{n,j}{\mathbb{Q}}_n^{(j)} +Y_{n,j} L^\text{T}_{n-1,j}{\mathbb{Q}}_{n-1}^{(j)}
+Z_{n,j} L^\text{T}_{n-2,j}{\mathbb{Q}}_{n-2}^{(j)}\,, \qquad n \geq 1, \quad j=1,2,
\end{equation}
where $ L^\text{T}_{n,j}$ is the transpose matrix of $L_{n,j}$ defined in (\ref{defLL}).

\section{Explicit expressions for algebraic properties in the monic case}\label{Sec:5}

As pointed out in the introduction, in this section we give the explicit expression of the coefficients $\widehat{G}_{n,n-1}$ and $\widehat{G}_{n,n-2}$ in~(\ref{EXPPMonico}) in terms of the coefficients $a$, $b_j$, $c_j$, $d_3$, $e$, $f_j$ fully characterizing the already mentioned partial differential equation~(\ref{Lysk1}). After that, results already given in the previous section will allow us to express (for monic polynomials) the three algebraic and differential properties here considered (\ref{RRTT}), (\ref{SRTRR}) and (\ref{POLDER1}), in terms of the admissible partial differential equation coefficients in (\ref{Lysk1}).
\begin{propo}\label{expcof}
Let $\widehat{\mathbb{P}}_n$ ($n\in \mathbb{N}_0$) be a monic vector polynomial, as given by the expansion~(\ref{EXPPMonico}), solution of an admissible hypergeometric type partial differential equation of the form~(\ref{Lysk1}). Then, the matrix coefficients $\widehat{G}_{n,n-1} \in \mathcal{M}^{(n+1,n)}$ and $\widehat{G}_{n,n-2}\in \mathcal{M}^{(n+1,n-1)}$ in (\ref{EXPPMonico}) can be written in term of the coefficients $a,b_j,c_j,d_3,e,f_j$ in~(\ref{Lysk1}) as:
\begin{align}\label{COEFPDE}
\widehat{G}_{n,n-1}&= \begin{pmatrix}
\tilde{g}_{1,1} & &  &  & \text{\circle{15}} \\
\tilde{g}_{2,1} & \tilde{g}_{2,2} & & &  \\
& \ddots & \ddots & &  \\
& & \tilde{g}_{n-1,n-2} & \tilde{g}_{n-1,n-1} & \\
& & & \tilde{g}_{n,n-1} & \tilde{g}_{n,n} \\
& \text{\circle{15}} & & 0 & \tilde{g}_{n+1,n}
\end{pmatrix}\,,\quad (n\geq 1)\,,
\end{align}
where, for $1 \leq i \leq n$,
\begin{align*}
\tilde{g}_{i,i}&=\frac{(n+1-i)((n-i)b_1+2(i-1)c_3+f_1)}{\varpi_{2n-2}}\,,\\
\tilde{g}_{i+1,i}&=\frac{i((i-1)b_2+2(n-i)b_3+f_2)}{\varpi_{2n-2}}\,,
\end{align*}
and
\begin{align}\label{COEFPDE2}
\widehat{G}_{n,n-2}&=\begin{pmatrix}
{g}_{1,1} & &  &  \text{\circle{15}}  \\
{g}_{2,1} & {g}_{2,2} &  &  \\
{g}_{3,1} & {g}_{3,2} & {g}_{3,3} &  \\
 \ddots& \ddots & \ddots & \\& \ddots & \ddots &\ddots \\
& {g}_{n-1,n-3}& {g}_{n-1,n-2}& {g}_{n-1,n-1}  \\
\phantom{aaa}\text{\circle{15}}&    & {g}_{n,n-2} & {g}_{n,n-1}\\
&   & 0 & {g}_{n+1,n-1}
\end{pmatrix}\,,\quad (n\geq 2)\,,
\end{align}
where, for $1 \leq i \leq n-1$,
\begin{align*}
{g}_{i,i}&=\frac{(n-i)(n+1-i)}{2\varpi_{2n-2}\varpi_{2n-3}} \\
&\times (\varpi_{2n-2}c_1+((n-i)b_1+2(i-1)c_3+f_1)((n-i-1)b_1+2(i-1)c_3+f_1)),\\
{g}_{i+1,i}& =\frac{i(n-i)}{\varpi_{2n-2}\varpi_{2n-3}}
(f_1f_2+d_3 \varpi_{2n-2} +b_3(2(n-2+2(i-2)(n-i-1))c_3 \\
&+(2n-2i-1)f_1)+(2i-1)c_3f_2+(i-1)b_2((2i-1)c_3+f_1) \\
& +(n-1-i)((i-1)b_2+(2n-2i-1)b_3+f_2)) ,\\
{g}_{i+2,i}&=\frac{i(i+1)}{2\varpi_{2n-2}\varpi_{2n-3}} \\
& \times
(\varpi_{2n-2}c_2+((i-1)b_2+2(n-i-1)b_3+f_2)(ib_2+2(n-i-1)b_3+f_2)),
\end{align*}
In all of these expressions $\varpi_{n}=na+e \neq 0$ as already shown in Eq.~(\ref{NONU}).
\end{propo}
\begin{proof} Plug into equation (\ref{Lysk1}) the expansion~(\ref{EXPPMonico}) and then make equal zero the coefficients of the column vector of monomials $\mathbf{x}^{k}$ (defined in (\ref{MONO})) for $k=n$, $n-1$, $n-2$.
\end{proof}

Now, having in mind that in the monic case $\widehat{G}_{n,n}= I_{n+1}$, from these results and Theorem \ref{TTRRenG}, we can deduce the following corollaries.

\begin{coro}[Three--term recurrence relations]\label{TTRRenEDP}
For monic polynomials and $j=1,2$, the coefficients of the three--term recurrence relation (\ref{RRTT}) are given in terms of the coefficients of the second order partial differential equation (\ref{Lysk1}) by
\begin{align}
A_{n,j}&=L_{n,j},\label{AAmonic} \\
B_{0,1}&=\left( -\frac{f_1}{e} \right), \quad B_{0,2}=\left( -\frac{f_2}{e} \right), \label{BB0monic} \\
B_{n,j}&=\widehat{G}_{n,n-1}L_{n-1,j}-L_{n,j}\widehat{G}_{n+1,n}, \quad n\geq 1, \label{BBmonic} \\
C_{1,1}&=\begin{pmatrix}  \frac{-c_1 e^2 + f_1 (b_1 e- a f_1)}{e^2
(a+e)} \\  \frac{-d_3 e^2 + b_3 e f_1 + c_3 e f_2 - a f_1 f_2}{e^2
(a+e)} \end{pmatrix} , \qquad
C_{1,2}=\begin{pmatrix} \frac{-d_3 f_1^2 + b_3 e f_1 + c_3 e f_2 - a f_1 f_2}{e^2 (a+e)} \\ \frac{-c_2 e^2 + f_2 (b_2 e - a f_2) }{e^2 (a+e)}  \end{pmatrix} \label{CC1monic} \\
C_{n,j}&=\widehat{G}_{n,n-2}L_{n-2,j}-L_{n,j}\widehat{G}_{n+1,n-1}-B_{n,j}\widehat{G}_{n,n-1},
\quad n \geq 2. \label{CCmonic}
\end{align}
\end{coro}

It has some interest to remark here that, as described in \cite{DUXU}, since
\begin{equation}
{\rm{rank}} (L_{n,j})=n+1= {\rm{rank}} (C_{n+1,j}), \quad j=1,2, \quad n \geq 0,
\end{equation}
the columns of the joint matrices
\[
L_n=\left(L_{n,1}^T \,,L_{n,2}^T \right)^T \quad \text{and}\quad C_n=\left(C_{n,1}^T \,,C_{n,2}^T \right)^T
\]
of size $(2n+2)\times(n+2)$ and $(2n+2)\times n$ respectively, are linearly independent, i.e.
\begin{equation}
{\rm{rank}} (L_{n})= n+2, \quad {\rm{rank}} (C_{n})= n.
\end{equation}
Therefore, the matrix $L_{n}$ has full rank so that there exists an unique matrix $D_n^{\dag}$ of size $(n+2)\times(2n+2)$, called the generalized inverse of $L_{n}$:
\begin{equation}
D_n^{\dag}= \left(D_{n,1} |D_{n,2} \right)=\left(L_n^TL_n\right)^{-1}L_n^T,
\end{equation}
such that
\[
D_n^{\dag}L_{n}=I_{n+2}.
\]

Moreover, using the left inverse $D_n^{\dag}$ of the joint matrix
$L_n$
\[ D_n^{\dag}=\begin{pmatrix} 1
& & && 0& & &&
\\&1/2&&\text{\circle{10}}&1/2&&\text{\circle{10}}&&\\&&\ddots&
&&\ddots&&\\& \text{\circle{10}}&& 1/2&0&&1/2&&\\&&&&0&&&1
\end{pmatrix},
\]
we can write a recursive formula for the monic orthogonal polynomials
\begin{equation}\label{RF}
\widehat{\mathbb{P}}_{n+1}=D_n^{\dag}\left[\begin{pmatrix} x
\\y \end{pmatrix}\otimes I_{n+1}-B_{n}\right]
{\widehat{\mathbb{P}}}_{n}-D_n^{\dag} C_{n}
\widehat{\mathbb{P}}_{n-1}, \quad n \geq 0,
\end{equation}
with the initial conditions $\widehat{\mathbb{P}}_{-1}=0$, $\widehat{\mathbb{P}}_{0}=1$, where $\otimes$ denotes the Kronecker product and
\begin{equation}\label{JOMA}
B_n=\left(B_{n,1}^T \,,B_{n,2}^T \right)^T\,, \quad
C_n=\left(C_{n,1}^T \,,C_{n,2}^T \right)^T\,,
\end{equation}
are matrices of size $(2n+2)\times(n+1)$ and $(2n+2)\times n$,
respectively. This recurrence (\ref{RF}) gives another
representation of \cite[(3.2.10)]{DUXU}, already presented in the
bivariate discrete case in \cite{RAG3}.

\begin{coro}[Structure relations]\label{strmonic}
For monic polynomials, $n \geq 3$ and $j=1,2$, with the notation (\ref{Eq:notationphi}) for $\phi_{j}(x,y)$ given in (\ref{NOT:1}), the coefficients of the structure relations  (\ref{SRTRR}) are given in terms of the coefficients of the second order partial differential equation (\ref{Lysk1}) by
\begin{align}
W_{n,j}&=\mathbb{E}_{n,j} \left( \alpha_{j} L_{n-1,1}L_{n,1} + \beta_{j} L_{n-1,2}L_{n,1} + \gamma_{j} L_{n-1,2} L_{n,2} \right), \label{WWmonic}  \\
S_{n,j}&=\mathbb{E}_{n,j} \left( \delta_{j} L_{n-1,1} + \varepsilon_{j} L_{n-1,2} \right)-W_{n,j} \widehat{G}_{n+1,n} \label{SS1monic} \\
 & +\widehat{G}_{n,n-1}\mathbb{E}_{n-1,j}\left(\alpha_{j} L_{n-2,1}L_{n-1,1}+\beta_{j} L_{n-2,2}L_{n-1,1} +\gamma_{j} L_{n-2,2} L_{n-1,2} \right), \nonumber \\
T_{n,j}&=\omega_{j} E_{n,j} + \widehat{G}_{n,n-1} \mathbb{E}_{n-1,j} \left( \delta_{j} L_{n-2,1} + \varepsilon_{j} L_{n-2,2} \right)-W_{n,j} \widehat{G}_{n+1,n-1} - S_{n,j} \widehat{G}_{n,n-1} \label{TT1monic} \\
 & +\widehat{G}_{n,n-2}\mathbb{E}_{n-2,j}\left(\alpha_{j} L_{n-3,1}L_{n-2,1}+\beta_{j} L_{n-3,2}L_{n-2,1} +\gamma_{j} L_{n-3,2} L_{n-2,2} \right) \nonumber .
\end{align}
\end{coro}

\begin{coro}[Derivative representations]\label{DRenEDP}
For monic polynomials, $n \geq 2$ and $j=1,2$, the coefficients of the derivative representations (\ref{POLDER1}) are given in terms of the coefficients of the second order partial differential equation (\ref{Lysk1}) by
\begin{align}
V_{n,j}&=(L_{n,j}\mathbb{E}_{n+1,j})^{-1}, \label{VV1monic} \\
Y_{n,j}&=(\widehat{G}_{n,n-1}-V_{n,j}L_{n,j}\widehat{G}_{n+1,n}\mathbb{E}_{n,j})V_{n-1,j}, \label{YY1monic} \\
Z_{n,j}&=(\widehat{G}_{n,n-2}-V_{n,j}L_{n,j}\widehat{G}_{n+1,n-1}\mathbb{E}_{n-1,j}-Y_{n,j}L_{n-1,j}\widehat{G}_{n,n-1}\mathbb{E}_{n-1,j})V_{n-2,j}. \label{ZZ1monic}
\end{align}
\end{coro}

The results obtained in Section \ref{Sec:4} and \ref{Sec:5} can be applied to any polynomial solution of an admissible partial differential equation of hypergeometric type. In the next section we study in detail the monic Appell polynomials \cite{APP}. The non--monic case is also briefly discussed.

\section{Example: Monic Appell polynomials}\label{Sec:6}

In this section we give explicit expressions of all the aforementioned matrices appea\-ring in the three--term recurrence
relations and structure relations, by applying the results obtained in the previous sections to the case of monic Appell polynomials. Other
examples of orthogonal polynomials of two variables can be treated in a similar way.

In 1882, Appell \cite{APP} introduced a three parameter family of polynomials of degree $n+m$ in terms of generalized Kamp\'{e} de F\'{e}riet hypergeometric
series \cite{SRIK} which in the monic case and reduced to two parameters can be written as \cite[Equation (12), p. 271]{ERDII}
\begin{multline}\label{ferietappell}
{\widehat{\mathrm{A}}}_{n,m}^{(\alpha,
\beta)}(x,y)=(-1)^{n+m}\frac{(\alpha)_{n}\,(\beta)_{m}}{(\alpha+\beta+n+m)_{n+m}}\,
\\ \times \feriet{1:1;1}{0:1;1}{\alpha+\beta+n+m}{-n;-m}{-}{\alpha;\beta}{x}{y},
\end{multline}
where $\alpha> 0$ and $\beta> 0$ and $(\alpha)_n$ denotes the Pochhammer symbol.

\subsection{Second order partial differential equation}
The admissible partial differential equation of hypergeometric type satisfied by ${\widehat{\mathrm{A}}}_{n,m}^{(\alpha, \beta)}(x,y)$ is \cite[Equation (15), p. 272]{ERDII}
\begin{multline}\label{equationappell}
x(1-x)  \frac{\partial^2 {\widehat{\mathrm{A}}}_{n,m}^{(\alpha,
\beta)}}{\partial x^2}- 2xy \frac{\partial^2
{\widehat{\mathrm{A}}}_{n,m}^{(\alpha, \beta)}}{\partial x
\partial y}+ y(1-y) \frac{\partial^2 {\widehat{\mathrm{A}}}_{n,m}^{(\alpha,
\beta)}}{\partial y^2}+ (\alpha-(\alpha+\beta+1)x) \frac{\partial
{\widehat{\mathrm{A}}}_{n,m}^{(\alpha, \beta)}}{\partial x }\\+
(\beta-(\alpha+\beta+1)y) \frac{\partial
{\widehat{\mathrm{A}}}_{n,m}^{(\alpha, \beta)}}{\partial
y}+(n+m)(\alpha+\beta+n+m) {\widehat{\mathrm{A}}}_{n,m}^{(\alpha,
\beta)}=0.
\end{multline}

From (\ref{derLysk1}) we obtain that the partial derivatives
\[
z^{(r,s)}(x,y) = \frac{\partial^{r+s}}{\partial x^{r}\,\partial y^{s}} {\widehat{\mathrm{A}}}_{n,m}^{(\alpha,
\beta)} (x,y),
\]
satisfy the admissible partial differential equation of hypergeometric type
\begin{multline*}
x(1-x)  \frac{\partial^2}{\partial x^2} z^{(r,s)}(x,y)
- 2xy \frac{\partial^2}{\partial x \partial y} z^{(r,s)}(x,y)
+ y(1-y) \frac{\partial^2}{\partial y^2} z^{(r,s)}(x,y) + \\
+ (\alpha+r-(\alpha+\beta+1+2(r+s))x) \frac{\partial}{\partial x } z^{(r,s)}(x,y) \\ +
(\beta+s-(\alpha+\beta+1+2(r+s))y) \frac{\partial }{\partial y} z^{(r,s)}(x,y) \\
+(n+m-r-s)(\alpha+\beta+n+m+r+s) z^{(r,s)}(x,y)=0.
\end{multline*}

It is easy to check that, from the above differential equation, the matrices $\widehat{G}_{n,n-1}$ and $\widehat{G}_{n,n-2}$ in the expansion
\begin{multline}\label{devAA}
{\widehat{{\mathbb{A}}}}_n={\widehat{{\mathbb{A}}}}_n^{(\alpha,
\beta)}(x,y) =({\widehat{\mathrm{A}}}_{n,0}^{(\alpha,
\beta)}(x,y),\dots, {\widehat{\mathrm{A}}}_{n-i,i}^{(\alpha,
\beta)}(x,y),\dots, {\widehat{\mathrm{A}}}_{0,n}^{(\alpha,
\beta)}(x,y))^\text{T} \\
= \textbf{x}^n+ \widehat{G}_{n,n-1}\textbf{x}^{n-1}+ \dots +
\widehat{G}_{n,0}\,\textbf{x}^0\,,
\end{multline}
are explicitly given from the general expressions (\ref{COEFPDE}) and (\ref{COEFPDE2}).

\subsection{Orthogonality}
From (\ref{WEIF}) we obtain the following weight function
\begin{equation}\label{pesoappell}
\varrho(x,y) = x^{\alpha-1}y^{\beta-1}\,.
\end{equation}
From (\ref{Eq:alpha}) the positivity of $\varrho(x,y)$ yields the following triangular domain
\begin{equation}\label{regionappell}
{\mathcal{R}}=\{(x,y): x>0,\; y>0,\; x+y< 1\}.
\end{equation}
Therefore, the orthogonality property reads as
\begin{equation}
\iint_{\mathcal{R}} \varrho(x,y)
{\widehat{\mathrm{A}}}_{n,m}^{(\alpha, \beta)}(x,y)
{\widehat{\mathrm{A}}}_{k,s}^{(\alpha, \beta)}(x,y)\,dx\,dy
=\Lambda_{N}^{(\alpha, \beta)} \delta_{N,K} \,,
\end{equation}
where $N=n+m$, $K=k+s$.

\subsection{Three term recurrence relations}
For $n\geq 0$, monic Appell polynomials satisfy the three term recurrence relations
\begin{align*}
x {\widehat{{\mathbb{A}}}}_{n}&= L_{n,1}
{\widehat{{\mathbb{A}}}}_{n+1} + B_{n,1}
{\widehat{{\mathbb{A}}}}_{n} + C_{n,1} {\widehat{{\mathbb{A}}}}_{n-1} \,, \\
y {\widehat{{\mathbb{A}}}}_{n}&= L_{n,2}
{\widehat{{\mathbb{A}}}}_{n+1} + B_{n,2}
{\widehat{{\mathbb{A}}}}_{n} + C_{n,2}
{\widehat{{\mathbb{A}}}}_{n-1} \,,
\end{align*}
with the initial conditions ${\widehat{{\mathbb{A}}}}_{0}=1$ and ${\widehat{{\mathbb{A}}}}_{-1}=0$, where ${\widehat{{\mathbb{A}}}}_n$ is defined in (\ref{devAA}) and $L_{n,j}$ are defined in (\ref{defLL}). Using (\ref{BB0monic}) and (\ref{BBmonic}) the recursion coefficients $B_{n,j}$ are given by
\begin{align}
{B}_{n,1}&=
\begin{pmatrix}
b_{0,0} & 0&  &  & \text{\circle{15}} \\
b_{1,0} & b_{1,1} & 0 & &  \\
& \ddots & \ddots & \ddots &  \\
& & b_{n-1,n-2}& b_{n-1,n-1} & 0 \\
& \text{\circle{15}}  & & b_{n,n-1} & b_{n,n}
\end{pmatrix}\,,
\end{align}
where
\begin{align*}
b_{i,i}&=-\frac{(n-i)(\alpha+n-1-i)}{2n-1+\alpha+\beta}+ \frac{(n+1-i)(\alpha+n-i)}{2n+1+\alpha+\beta},\quad 0 \leq i \leq n,\\
b_{i+1,i}&=-\frac{2(i+1)(\beta+i)}{(2n-1+\alpha+\beta)(2n+1+\alpha+\beta)}
,\quad 0 \leq i \leq n-1,
\end{align*}
and
\begin{equation}
{B}_{n,2}=
\begin{pmatrix}
\tilde{b}_{0,0} & \tilde{b}_{0,1}&  &  & \text{\circle{15}} \\
0 & \tilde{b}_{1,1} &\tilde{b}_{1,2}  & &  \\
& \ddots & \ddots & \ddots &  \\
& & & \tilde{b}_{n-1,n-1} & \tilde{b}_{n-1,n} \\
& \text{\circle{15}}  & & 0 & \tilde{b}_{n,n}
\end{pmatrix}\,,
\end{equation}
with
\begin{align*}
{\tilde{b}}_{i,i}&=1+ \frac{i(2n-i+\alpha)}{2n-1+\alpha+\beta}-\frac{(i+1)(\alpha+2n+1-i)}{2n+1+\alpha+\beta},\quad 0 \leq i \leq n,\\
{\tilde{b}}_{i,i+1}& =-
\frac{2(n-i)(\alpha+n-1-i)}{(2n-1+\alpha+\beta)(2n+1+\alpha+\beta)},\quad
0 \leq i \leq n-1.
\end{align*}
Moreover, using (\ref{CC1monic}) and (\ref{CCmonic}) we have
\begin{equation}
{C}_{n,1}=
\begin{pmatrix}
{c}_{0,0} & & & &  \text{\circle{15}}  \\
{c}_{1,0} & {c}_{1,1} &  & &  \\
 {c}_{2,0}& {c}_{2,1} & {c}_{2,2} & & \\
& \ddots & \ddots & \ddots&\\
&\text{\circle{15}} &{c}_{n-1,n-3}& {c}_{n-1,n-2}& {c}_{n-1,n-1}  \\
&   & & {c}_{n,n-2} & {c}_{n,n-1}
\end{pmatrix}\,,
\end{equation}
where
\begin{align*}
{c}_{i,i}&=\frac{(n-i)(\alpha+n-1-i)(n+i+\beta)(n-1+i+\alpha+\beta)}{(2n+\alpha+\beta)(2n-1+\alpha+\beta)^2(2n-2+\alpha+\beta)},\quad 0 \leq i \leq n-1,\\
{c}_{i+1,i}& =- \frac{(i+1)(\beta+i)(2(n-i-1)(n+i+\beta)+\alpha(2n+\alpha+\beta-2))}{(2n+\alpha+\beta)(2n-1+\alpha+\beta)^2(2n-2+\alpha+\beta)} ,\quad 0 \leq i \leq n-1,\\
{c}_{i+2,i}&
=\frac{(i+2)(i+1)(\beta+i)(\beta+i+1)}{(2n+\alpha+\beta)(2n-1+\alpha+\beta)^2(2n-2+\alpha+\beta)}
,\quad 0 \leq i \leq n-2,
\end{align*}
and
\begin{equation}
{C}_{n,2}=
\begin{pmatrix}
{\tilde{c}}_{0,0} & {\tilde{c}}_{0,1}&  &  & \text{\circle{15}}  \\
{\tilde{c}}_{1,0} & {\tilde{c}}_{1,1} & {\tilde{c}}_{1,2} & & \\
& \ddots & \ddots & \ddots &\\
& {\tilde{c}}_{n-2,n-3}& {\tilde{c}}_{n-2,n-2}& {\tilde{c}}_{n-2,n-1}  \\
& & {\tilde{c}}_{n-1,n-2}& {\tilde{c}}_{n-1,n-1}  \\
& \text{\circle{15}}   &  & {\tilde{c}}_{n,n-1}
\end{pmatrix}\,,
\end{equation}
with
\begin{align*}
{\tilde{c}}_{i,i}&=-\frac{(n-i)(\alpha+n-1-i)(\beta(2n-2+\beta)+\alpha(2i+\beta)+2i(2n-1-i))}{(2n+\alpha+\beta)(2n-1+\alpha+\beta)^2(2n-2+\alpha+\beta)},\\
{\tilde{c}}_{i+1,i}& =\frac{(i+1)(\alpha+2n-1-i)(\beta+i)(\alpha+\beta+2n-2-i)}{(2n+\alpha+\beta)(2n-1+\alpha+\beta)^2(2n-2+\alpha+\beta)},
\end{align*}
for $0 \leq i \leq n-1$ and
\[
{\tilde{c}}_{i,i+1}=\frac{(n-i)(n-1-i)(\alpha+n-1-i)(\alpha+n-2-i)}{(2n+\alpha+\beta)(2n-1+\alpha+\beta)^2(2n-2+\alpha+\beta)},
\]
for $0 \leq i \leq n-2$.

\subsection{First structure relations}
The partial differential equation (\ref{equationappell}) for monic Appell polynomials corresponds to cases (vi), (ix) and (x) in Section \ref{Sec:3}. Therefore, we obtain that
\begin{equation}\label{Eq:phiappell}
\phi^{(r,s)}(x,y)= [x(1-x-y)]^{r}\, [y(1-x-y)]^{s}.
\end{equation}

As a consequence, the structure relations (\ref{SRTRR}) satisfied by monic Appell polynomials defined in (\ref{devAA}) are given by
\begin{align}\label{SRAPP}
x(1-x-y)\frac{\partial}{\partial
x} {\widehat{\mathbb{A}}}_n&=W_{n,1} {\widehat{\mathbb{A}}}_{n+1}+S_{n,1} {\widehat{\mathbb{A}}}_{n}+T_{n,1} {\widehat{\mathbb{A}}}_{n-1},\\
y(1-x-y)\frac{\partial}{\partial y}
{\widehat{\mathbb{A}}}_n&=W_{n,2} {\widehat{\mathbb{A}}}_{n+1}+S_{n,2} {\widehat{\mathbb{A}}}_{n}+T_{n,2} {\widehat{\mathbb{A}}}_{n-1},
\end{align}
for $n \geq 1$, where using (\ref{WWmonic}) we get
\begin{equation}\label{EC:MATWN1}
{W}_{n,1}=
\begin{pmatrix}
{w}_{0,0} &{w}_{0,1} & 0 &  & \text{\circle{15}} \\
 0& {w}_{1,1} & {w}_{1,2} & &  \\
& \ddots  & \ddots & &  \\
& & {w}_{n-1,n-1}& {w}_{n-1,n} & 0 \\
& \text{\circle{15}}  & & {w}_{n,n}& {w}_{n,n+1}
\end{pmatrix}\,,
\end{equation}
with ${w}_{i,i}=-n+i={w}_{i,i+1}$, $0 \leq i \leq n$, and
\begin{equation}\label{EC:MATWN2}
{W}_{n,2}=
\begin{pmatrix}
\tilde{w}_{0,0} &\tilde{w}_{0,1} & 0 &  & \text{\circle{15}} \\
0 & \tilde{w}_{1,1} & \tilde{w}_{1,2} & &  \\
&  & \ddots & \ddots &  \\
& & \tilde{w}_{n-1,n-1}& \tilde{w}_{n-1,n} & 0 \\
& \text{\circle{15}}  & & \tilde{w}_{n,n} & \tilde{w}_{n,n+1}
\end{pmatrix}\,,
\end{equation}
with $\tilde{w}_{i,i}=-i=\tilde{w}_{i,i+1}$, $0 \leq i \leq n$.

Moreover, from (\ref{SS1monic}) we obtain
\begin{equation}\label{EC:MATSN1}
{S}_{n,1}=
\begin{pmatrix}
{s}_{0,0} & {s}_{0,1}&  &  \text{\circle{15}}  \\
{s}_{1,0} & {s}_{1,1} & {s}_{1,2} &  \\
 &  \ddots & \ddots &\\
&{s}_{n-1,n-2} & {s}_{n-1,n-1}& {s}_{n-1,n}  \\
& \text{\circle{15}}   & 0 & 0
\end{pmatrix}\,,
\end{equation}
with
\begin{align*}
{s}_{i,i}&=-\frac{(n-i)(-n+(2n-1)i-4i^2 + (n-2-3i) \beta + \alpha (n-1+i+\alpha+\beta))}{(2n-1+\alpha+\beta)(2n+1+\alpha+\beta)}, \\
{s}_{i,i+1}& =-\frac{(n-i)(n-1-i+\alpha)(2i+1+\alpha+\beta)}{(2n-1+\alpha+\beta)(2n+1+\alpha+\beta)},
\end{align*}
for $0 \leq i \leq n-1$, and
\[
{s}_{i+1,i} =\frac{2(i+1)(n-1-i)(\beta+i)}{(2n-1+\alpha+\beta)(2n+1+\alpha+\beta)},
\]
for $\quad 0 \leq i \leq n-2$. Also, using (\ref{SS1monic}) we have
\begin{equation}\label{EC:MATSN2}
{S}_{n,2}=
\begin{pmatrix}
0 & 0&  &  & \text{\circle{15}} \\
\tilde{s}_{1,0} & \tilde{s}_{1,1} &\tilde{s}_{1,2}  & &  \\
0 & \tilde{s}_{2,1} &\tilde{s}_{2,2}  & &  \\
& \ddots & \ddots & \ddots &  \\
& &\tilde{s}_{n-1,n-2} & \tilde{s}_{n-1,n-1} & \tilde{s}_{n-1,n} \\
& \text{\circle{15}}  & 0& \tilde{s}_{n,n-1} & \tilde{s}_{n,n}
\end{pmatrix}\,,
\end{equation}
with
\[
\tilde{s}_{i,i}=\frac{i(\beta - {\beta}^2 - i + \beta\,i + 4\,i^2 -
  \alpha\,\left( -2 + \beta + 3\,i - 2\,n \right)  -
  2\,\left( -1 + \beta + 3\,i \right) \,n + 2\,n^2)}{(2n+1+\alpha+\beta)(2n-1+\alpha+\beta)},\\
\]
for $1 \leq i \leq n$, and
\begin{align*}
\tilde{s}_{i+1,i}&=- \frac{\left( 1 + i \right) \,\left( \beta + i \right) \,
      \left( -1 + \alpha + \beta - 2\,i + 2\,n \right) }{
      \left( -1 + \alpha + \beta + 2\,n \right) \,
      \left( 1 + \alpha + \beta + 2\,n \right) },\\
\tilde{s}_{i,i+1}&=\frac{2\,i\,\left( -i + n \right) \,
    \left( -1 + \alpha - i + n \right) }{\left( -1 + \alpha +
      \beta + 2\,n \right) \,
    \left( 1 + \alpha + \beta + 2\,n \right) },\\
\end{align*}
for $0 \leq i \leq n-1$. Furthermore, from (\ref{TT1monic}) it holds
\begin{equation}\label{EC:MATTN1}
{T}_{n,1}=
\begin{pmatrix}
{t}_{0,0} & {t}_{0,1}&  & &  \text{\circle{15}}  \\
{t}_{1,0} & {t}_{1,1} & {t}_{1,2} &  &\\
{t}_{2,0} & {t}_{2,1} & {t}_{2,2} & {t}_{2,3}& \\
& \ddots & \ddots & \ddots  &\\
& {t}_{n-2,n-4}& {t}_{n-2,n-3}& {t}_{n-2,n-2}& {t}_{n-2,n-1}  \\
& & {t}_{n-1,n-3}& {t}_{n-1,n-2}& {t}_{n-1,n-1}  \\
& \text{\circle{15}}&   & 0 & 0
\end{pmatrix}\,,
\end{equation}
where
\begin{align*}
{t}_{i,i}&=  \frac{\left( n-i \right) \,\left( n-1 + \alpha - i
\right) }
  {\left( -2 + \alpha + \beta + 2\,n \right) \,
    {\left( -1 + \alpha + \beta + 2\,n \right) }^2\,
    \left( \alpha + \beta + 2\,n \right) } \\ &\times \left( {\beta}^2\,\left( 1 + i \right)  +
  i^2\,\left( 1 + 3\,i \right)  +
  \alpha\,\beta\,\left( 1 + n \right)  +
  {\alpha}^2\,\left( -i + n \right)  \right. \\ & \left. +
  \beta\,\left( i\,\left( 3 + 4\,i \right)  +
     n\,\left( -2\,i + n \right)  \right)  +
  n\,\left( -\left( \left( -2 + i \right) \,i \right)  +
     n\,\left( -1 - i + n \right)  \right)  \right. \\ & \left. +
  \alpha\,\left( i\,\left( 2 + i \right)  +
     n\,\left( -1 - 2\,i + 2\,n \right)  \right) \right) ,\quad 0 \leq i \leq n-1,\\
{t}_{i,i+1}& = \frac{(n-i)(n-i-1)(\alpha+n-2-i)(\alpha+n-i-1)(\alpha+\beta+n+i)}{(2n+\alpha+\beta)(2n-1+\alpha+\beta)^{2}(2n-2+\alpha+\beta)},\,\, 0 \leq i \leq n-2,\\
{t}_{i+1,i}& =  \frac{(\beta+i)(n-i-1)(i+1)}{(2n+\alpha+\beta)(2n-1+\alpha+\beta)^{2}(2n-2+\alpha+\beta)}
\\ & \times (\alpha(\alpha+\beta+n+i-1)+\beta(n-2i-3)+(-2+(2n-5)i-3i^{2})) ,\,\, 0 \leq i \leq n-2,\\
{t}_{i+2,i}&=   -\frac{(\beta+i)(\beta+i+1)(n-i-2)(i+1)(i+2)}{(2n+\alpha+\beta)(2n-1+\alpha+\beta)^{2}(2n-2+\alpha+\beta)},\quad 0 \leq i \leq n-3,
\end{align*}
and using (\ref{TT1monic})
\begin{equation}\label{EC:MATTN2}
{T}_{n,2}=
\begin{pmatrix}
0 & 0&  &  & \text{\circle{15}} \\
\tilde{t}_{1,0} & \tilde{t}_{1,1} &\tilde{t}_{1,2}  & &  \\
\tilde{t}_{2,0} & \tilde{t}_{2,1} &\tilde{t}_{2,2}  &\tilde{t}_{2,3} &  \\
& \ddots & \ddots & \ddots &  \\
& \tilde{t}_{n-2,n-4}& \tilde{t}_{n-2,n-3} & \tilde{t}_{n-2,n-2} & \tilde{t}_{n-2,n-1} \\
& &\tilde{t}_{n-1,n-3} & \tilde{t}_{n-1,n-2} & \tilde{t}_{n-1,n-1} \\
& \text{\circle{15}}  & & \tilde{t}_{n,n-2} &\tilde{t}_{n,n-1}
\end{pmatrix}\,,
\end{equation}
where
\begin{align*}
\tilde{t}_{i,i}&=\frac{i\,\left( -i + n \right) \,
    \left( -1 + \alpha - i + n \right) }{\left( -2 + \alpha +
      \beta + 2\,n \right) \,
    {\left( -1 + \alpha + \beta + 2\,n \right) }^2\,
    \left( \alpha + \beta + 2\,n \right) } \\ & \times
(\alpha\,\beta + {\beta}^2 - i\,\left( 1 + 3\,i - 4\,n \right)  -
  \beta\,\left( 2 + i - 2\,n \right)  +
  \alpha\,\left( -1 + 2\,i - n \right)  - n\,\left( 1 + n \right)) ,\\ & 1 \leq i \leq n-1,\\
\tilde{t}_{i,i+1}&=-\frac{i\,\left( -1 - i + n \right) \,
      \left( -i + n \right) \,
      \left( -2 + \alpha - i + n \right) \,
      \left( -1 + \alpha - i + n \right) }{\left( -2 + \alpha +
        \beta + 2\,n \right) \,
      {\left( -1 + \alpha + \beta + 2\,n \right) }^2\,
      \left( \alpha + \beta + 2\,n \right) } ,\quad 1 \leq i \leq n-2,\\
\tilde{t}_{i+1,i}&=\frac{\left( 1 + i \right) \,\left( \beta + i \right) }
  {\left( -2 + \alpha + \beta + 2\,n \right) \,
    {\left( -1 + \alpha + \beta + 2\,n \right) }^2\,
    \left( \alpha + \beta + 2\,n \right) } \\ & \times \left( -3\,{\left( 1 + i \right) }^3 +
  \left( 1 + n \right) \,\left( \alpha + n \right) \,
   \left( \alpha + \beta + 2\,n \right)  +
  {\left( 1 + i \right) }^2\,
   \left( 1 + 4\,\alpha + \beta + 8\,n \right) \right. \\ & \left. -
  \left( 1 + i \right) \,\left( \alpha\,
      \left( 3 + \alpha \right)  -
     \left( -2 + \beta \right) \,\beta + 4\,n + 6\,\alpha\,n +
     6\,n^2 \right) \right) ,\quad 0 \leq i \leq n-1,\\
\tilde{t}_{i+2,i}&=\frac{\left( 1 + i \right) \,\left( 2 + i \right) \,
    \left( \beta + i \right) \,\left( 1 + \beta + i \right) \,
    \left( -2 + \alpha + \beta - i + 2\,n \right) }{\left( -2 +
      \alpha + \beta + 2\,n \right) \,
    {\left( -1 + \alpha + \beta + 2\,n \right) }^2\,
    \left( \alpha + \beta + 2\,n \right) },\quad
0 \leq i \leq n-2.
\end{align*}

\subsection{Derivative representations or second structure relations}
The monic Appell poly\-no\-mials defined in (\ref{devAA}) satisfy the derivative representations
\begin{equation}\label{POLDERAPP}
{\widehat{\mathbb{A}}}_n=V_{n,1}{\mathbb{Q}}_{n}^{(j)}+Y_{n,1}{\mathbb{Q}}_{n-1}^{(j)}+Z_{n,1}{\mathbb{Q}}_{n-2}^{(j)},\quad n \geq 2, \quad j=1,2,
\end{equation}
where ${\mathbb{Q}}_{n}^{(1)}=L_{n,1}\frac{\partial}{\partial x} {\widehat{{\mathbb{A}}}}_{n+1}$ and ${\mathbb{Q}}_{n}^{(2)}=L_{n,2}\frac{\partial}{\partial y} {\widehat{{\mathbb{A}}}}_{n+1}$. In this example the matrices $V_{n,j}$ defined in (\ref{VV1monic}) are given by
\begin{equation}\label{EC:MATVN1}
{V}_{n,1}=
\begin{pmatrix}
{v}_{0,0} & &  &  & \text{\circle{15}} \\
 & {v}_{1,1} &  & &  \\
&  & \ddots & &  \\
& & & {v}_{n-1,n-1} &  \\
& \text{\circle{15}}  & & & {v}_{n,n}
\end{pmatrix}\,,\quad \text{with}\quad {v}_{i,i}= \frac{1}{n+1-i},\quad 0 \leq i \leq
n\,,
\end{equation}
and
\begin{equation}\label{EC:MATVN2}
{V}_{n,2}=
\begin{pmatrix}
\tilde{v}_{0,0} & &  &  & \text{\circle{15}} \\
 & \tilde{v}_{1,1} &  & &  \\
&  & \ddots & &  \\
& & & \tilde{v}_{n-1,n-1} &  \\
& \text{\circle{15}}  & & 0 & \tilde{v}_{n,n}
\end{pmatrix}\,,\quad \text{with} \quad \tilde{v}_{i,i}=\frac{1}{i+1},\quad 0 \leq i \leq
n\,.
\end{equation}

Moreover, from (\ref{YY1monic}) we have
\begin{equation}\label{EC:MATYN1}
{Y}_{n,1}=
\begin{pmatrix}
{y}_{0,0} & &  &  \text{\circle{15}}  \\
{y}_{1,0} & {y}_{1,1} &  &  \\
 &  \ddots & \ddots &\\
& & {y}_{n-1,n-2}& {y}_{n-1,n-1}  \\
& \text{\circle{15}}   & 0 & {y}_{n,n-1}
\end{pmatrix}\,,
\end{equation}
where
\begin{align*}
{y}_{i,i}&=\frac{2i+1-\alpha+\beta}{(2n+1+\alpha+\beta)(2n-1+\alpha+\beta)},\\
{y}_{i+1,i}&=-\frac{2(i+1)(\beta+i)}{(n-i)(2n+1+\alpha+\beta)(2n-1+\alpha+\beta)}
,\quad 0 \leq i \leq n-1,
\end{align*}
and
\begin{equation}\label{EC:MATYN2}
{Y}_{n,2}=
\begin{pmatrix}
\tilde{y}_{0,0} & &  &  & \text{\circle{15}} \\
\tilde{y}_{1,0} & \tilde{y}_{1,1} &  & &  \\
& \ddots & \ddots &  &  \\
& & \tilde{y}_{n-2,n-3} & \tilde{y}_{n-2,n-2} &  \\
& & & \tilde{y}_{n-1,n-2} & \tilde{y}_{n-1,n-1} \\
& \text{\circle{15}}  & & 0 & \tilde{y}_{n,n-1}
\end{pmatrix}\,,
\end{equation}
where
\begin{align*}
\tilde{y}_{i,i}&=-\frac{2(n-i)(n-1-i+\alpha)}{(1+i)(2n+1+\alpha+\beta)(2n-1+\alpha+\beta)},\\
 \tilde{y}_{i+1,i}&= \frac{2n-1-2i+\alpha-\beta}{(2n+1+\alpha+\beta)(2n-1+\alpha+\beta)},\quad 0 \leq i \leq n-1.
\end{align*}
Also, from (\ref{ZZ1monic}) we have
\begin{equation}\label{EC:MATZN1}
{Z}_{n,1}=
\begin{pmatrix}
{z}_{0,0} & &  &  \text{\circle{15}}  \\
{z}_{1,0} & {z}_{1,1} &  &  \\
{z}_{2,0} & {z}_{2,1} & {z}_{2,2} &  \\
\ddots& \ddots & \ddots & \\
& {z}_{n-2,n-4}& {z}_{n-2,n-3}& {z}_{n-2,n-2}  \\
& & {z}_{n-1,n-3}& {z}_{n-1,n-2}  \\
& \text{\circle{15}}   & 0 & {z}_{n,n-2}
\end{pmatrix}\,,
\end{equation}
where for $0 \leq i \leq n-2$,
\begin{align*}
{z}_{i,i}&=-\frac{(n-i)(n-1-i+\alpha)(n+i+\beta)}{(2n+\alpha+\beta)(2n-1+\alpha+\beta)^2(2n-2+\alpha+\beta)},\\
{z}_{i+1,i}& =\frac{(i+1)(-2(i+1)+\alpha-\beta)(\beta+i)}{(2n+\alpha+\beta)(2n-1+\alpha+\beta)^2(2n-2+\alpha+\beta)},\\
{z}_{i+2,i}&=\frac{(i+1)(i+2)(\beta+i)(\beta+i+1)}{(n-1-i)(2n+\alpha+\beta)(2n-1+\alpha+\beta)^2(2n-2+\alpha+\beta)},\\
\end{align*}
and
\begin{equation}\label{EC:MATZN2}
{Z}_{n,2}=
\begin{pmatrix}
\tilde{z}_{0,0} & &  &  \text{\circle{15}}  \\
\tilde{z}_{1,0} & \tilde{z}_{1,1} &  &  \\
\tilde{z}_{2,0} & \tilde{z}_{2,1} & \tilde{z}_{2,2} &  \\
\ddots & \ddots & \ddots & \\
 & \tilde{z}_{n-2,n-4}& \tilde{z}_{n-2,n-3}& \tilde{z}_{n-2,n-2}  \\
& & \tilde{z}_{n-1,n-3}& \tilde{z}_{n-1,n-2}  \\
& \text{\circle{15}}   & 0 & \tilde{z}_{n,n-2}
\end{pmatrix}\,,
\end{equation}
where for $0 \leq i \leq n-2$,
\begin{align*}
\tilde{z}_{i,i}&=\frac{\left( -1 - i + n \right) \,\left( -i + n \right) \,
    \left( -2 + \alpha - i + n \right) \,
    \left( -1 + \alpha - i + n \right) }{\left( 1 + i \right) \,
    \left( -2 + \alpha + \beta + 2\,n \right) \,
    {\left( -1 + \alpha + \beta + 2\,n \right) }^2\,
    \left( \alpha + \beta + 2\,n \right) },\\
\tilde{z}_{i+1,i}& =-\left( \frac{\left( -1 - i + n \right) \,
      \left( -2 + \alpha - i + n \right) \,
      \left( \alpha - \beta + 2\,\left( -1 - i + n \right)
        \right) }{\left( -2 + \alpha + \beta + 2\,n \right) \,
      {\left( -1 + \alpha + \beta + 2\,n \right) }^2\,
      \left( \alpha + \beta + 2\,n \right) } \right),\\
 \tilde{z}_{i+2,i}&=- \frac{\left( 2 + i \right) \,
      \left( 1 + \beta + i \right) \,
      \left( -2 + \alpha - i + 2\,n \right) }{\left( -2 +
        \alpha + \beta + 2\,n \right) \,
      {\left( -1 + \alpha + \beta + 2\,n \right) }^2\,
      \left( \alpha + \beta + 2\,n \right) }.
\end{align*}

\subsection{Non--monic orthogonal solutions of (\ref{equationappell})}

For any non--monic orthogonal polynomial solution of the partial differential equation (\ref{equationappell}) it is possible to obtain the main differential and algebraic properties by using the results given in Section \ref{Sec:4}. In this Section we give some relations for two concrete non--monic solutions of  (\ref{equationappell}), orthogonal with respect to (\ref{pesoappell}) in the domain (\ref{regionappell}).

On one hand, also in 1882, Appell considered a family of non--monic polynomials solution of the partial differential equation (\ref{equationappell}). This orthogonal family can obtained from the Rodrigues formula (\ref{Eq:Rodrigues}) (see \cite[Equation (11), p. 271]{ERDII}) using the weight (\ref{Eq:phiappell})
\begin{equation}\label{Eq:108}
F_{n,m}^{(\alpha,\beta)}(x,y)
=\frac{x^{1-\alpha} y^{1-\beta}}{(\alpha)_{n} \, (\beta)_{m}} \frac{\partial^{n+m}}{\partial x^{n} \, \partial y^{m}} \left[ x^{n+\alpha-1} \, y^{m+\beta-1} \, (1-x-y)^{n+m} \,\right] \,.
\end{equation}
These polynomials can also be obtained from the classical Appell's orthogonal polynomials defined in \cite[Equation (6), p. 63]{SUE} by taking $\gamma =\alpha+\beta$.

Clearly, both, monic Appell polynomials defined in (\ref{ferietappell}) and the non--monic family (\ref{Eq:108}), form a biorthogonal system in the domain (\ref{regionappell}) with respect to the weight function (\ref{pesoappell}) \cite[Equation (17), p. 272]{ERDII})
\[
\iint_{{\mathcal{R}}} x^{\alpha-1} y^{\beta-1}  F_{n,m}^{(\alpha,\beta)}(x,y)  {\widehat{\mathrm{A}}}_{k,l}^{(\alpha,\beta)}(x,y) \,{\rm{d}}x\,{\rm{d}}y = \delta_{nk} \delta_{ml} \Lambda_{n,m}\,, \quad \alpha,\beta>0.
\]

Then, if we denote
\begin{multline}\label{devFF}
{\mathbb{F}}_n={\mathbb{F}}_n^{(\alpha,\beta)}(x,y)
=(\mathrm{F}_{n,0}^{(\alpha, \beta)}(x,y),\dots, \mathrm{F}_{n-i,i}^{(\alpha, \beta)}(x,y),\dots, \mathrm{F}_{0,n}^{(\alpha, \beta)}(x,y))^\text{T} \\
= G_{n,n}^{F} \textbf{x}^n+ G_{n,n-1}^{F}\textbf{x}^{n-1}+ G_{n,n-2}^{F}\textbf{x}^{n-2} +\cdots + G_{n,0}^{F}\,\textbf{x}^0\,,
\end{multline}
we have the following formula linking both solutions in column polynomial vector form
\begin{equation}\label{matrixconnection1}
{\mathbb{F}}_n = G_{n,n}^{F} \widehat{{\mathbb{A}}}_n,
\end{equation}
where $\widehat{{\mathbb{A}}}_n$ is defined in (\ref{devAA}) and the entries of the matrix $G_{n,n}^{F}=(g_{i,j}^{F}(n))$ of size
$(n+1)\times(n+1)$ have the following explicit form
\[
g_{i,j}^{F}(n) = (-1)^{n} \binom{n}{j} \frac{(\alpha+n-i)_{n-j}\,(\beta+i)_{j}}{(\alpha)_{n-j} \, (\beta)_{j}}\,, \qquad 0 \leq i,j \leq n\,.
\]
Once the matrix $G_{n,n}^{F}$ is known, applying the formulae given in Section \ref{Sec:4} it is possible to obtain the coefficients of the thee--term recurrence relations, structure relations and derivative representations for this non--monic family (\ref{Eq:108}).

On the other hand, let us consider the non--monic Koornwinder family defined in \cite[Section 2.4.2, p. 86]{DUXU} with $\alpha\to \alpha-1/2$, $\beta \to \beta-1/2$ and $\gamma=1/2$:
\begin{equation}\label{Eq:KK}
K_{n,m}^{(\alpha,\beta)}(x,y)=P_{n}^{(2m+\beta,\alpha-1)}(2x-1)(1-x)^m P_{m}^{(0,\beta-1)}(\frac{2y}{1-x}-1), \qquad \alpha,\beta>0.
\end{equation}
As already mentioned, this family is another polynomial solution of the partial differential equation (\ref{equationappell}) orthogonal with respect to (\ref{pesoappell}) in the domain (\ref{regionappell}), where $P_{n}^{(\alpha,\beta)}(x)$ are the Jacobi polynomials \cite{KoeSwa94}. Therefore, monic and non--monic Appell polynomials and this family form also biorthogonal systems.  If we denote
\begin{multline*}
{\mathbb{K}}_n={\mathbb{K}}_n^{(\alpha, \beta)}(x,y)
=(\mathrm{K}_{n,0}^{(\alpha, \beta)}(x,y),\dots, \mathrm{K}_{n-i,i}^{(\alpha, \beta)}(x,y),\dots, \mathrm{K}_{0,n}^{(\alpha, \beta)}(x,y))^\text{T} \\
= G_{n,n}^{K} \textbf{x}^n+ G_{n,n-1}^{K}\textbf{x}^{n-1}+ G_{n,n-2}^{K}\textbf{x}^{n-2} +\cdots + G_{n,0}^{K}\,\textbf{x}^0\,,
\end{multline*}
then we have
\begin{equation}\label{matrixconnection2}
{\mathbb{K}}_n = G_{n,n}^{K} \widehat{{\mathbb{A}}}_n,
\end{equation}
where the entries of the matrix $G_{n,n}^{K}=(g_{i,j}^{K}(n))$ of size $(n+1)\times(n+1)$ have the following explicit form
\[
g_{i,j}^{K}(n) = \begin{cases}
0, & i<j \, , \\
\displaystyle{\frac{(\alpha+\beta+n+i)_{n-i}(\beta+j)_{i}}{(n-i)!\,j!\,(i-j)!}}, & i \geq j \,,
\end{cases}
\]
for $0 \leq i,j \leq n$.

Relations (\ref{matrixconnection1}) and (\ref{matrixconnection2}) between monic (\ref{ferietappell}) and non--monic (\ref{Eq:108}) Appell polynomials, and between monic Appell (\ref{ferietappell}) and Koornwinder (\ref{Eq:KK}) polynomials solve the following connection problems
\begin{align*}
F_{n-\ell,\ell}^{(\alpha,\beta)}(x,y)&= \sum_{j=0}^{n} g_{\ell,j}^{F}(n) {\widehat{\mathrm{A}}}_{n-j,j}^{(\alpha,\beta)}(x,y), \quad 0 \leq \ell \leq n, \\
K_{n-\ell,\ell}^{(\alpha,\beta)}(x,y)&= \sum_{j=0}^{n} g_{\ell,j}^{K}(n) {\widehat{\mathrm{A}}}_{n-j,j}^{(\alpha,\beta)}(x,y), \quad 0 \leq \ell \leq n,
\end{align*}
between the corresponding scalar orthogonal polynomial families.

Finally, we would like to mention here that our goal is not to exploit all possible situations covered by our approach, but to emphasize its systematic character, which allow one to implement it in any computer algebra system, here the \emph{Mathematica} \cite{WOL} symbolic language has been used.

\section*{Acknowledgments}

This work was partially supported by Ministerio de Educaci\'on y Ciencia of Spain under grants MTM2006--07186, MTM2009--14668--C02--01, MTM2009--14668--C02--02, and MTM2008--06689-C02, cofinanced by the European Community fund FEDER. A. Ronveaux also thanks the Departamento de Matem\'{a}tica Aplicada II of Universidade de Vigo for the kind invitations and financial support.

\end{document}